\documentclass[12pt,reqno]{amsart}

\usepackage{amsmath,amsfonts,amsthm,amscd,amssymb,graphicx,mathrsfs,slashed,cancel}
\usepackage[utf8]{inputenc}
\usepackage{amsbsy}
\usepackage{graphicx}
\usepackage{subcaption}
\usepackage[unicode=true,bookmarks=false,breaklinks=false,pdfborder={0 0 1},colorlinks=true]{hyperref}
\usepackage{bm}
\usepackage{tikz}
\usepackage[capitalise]{cleveref}
\usetikzlibrary{arrows}
\usepackage{wrapfig}
\usepackage[margin=1in]{geometry}
\usepackage{setspace}
\usepackage{todonotes}
\usepackage[all]{xy}
\usepackage{sseq}
\usepackage{cite}
\usepackage[normalem]{ulem}

\definecolor{skyblue}{rgb}{0.85,0.85,1}

\numberwithin{equation}{section}
\numberwithin{figure}{section}

\newtheorem{theorem}{Theorem}[section]
\newtheorem{lemma}[theorem]{Lemma}
\newtheorem{prop}[theorem]{Proposition}
\newtheorem{cor}[theorem]{Corollary}

\theoremstyle{definition}

\newtheorem{define}[theorem]{Definition}
\crefname{define}{definition}{definitions}

\theoremstyle{remark}
\newtheorem{rem}[theorem]{Remark}

\DeclareMathOperator{\dom}{dom}
\DeclareMathOperator{\supp}{supp}

\newcommand{\bbN}{\mathbb{N}}
\newcommand{\bbR}{\mathbb{R}}

\newcommand{\bbZ}{\mathbb{Z}}

\renewcommand{\Im}{\operatorname{Im}}

\newcommand{\cP}{\mathcal{P}}

\newcommand{\pO}{\partial \Omega}
\newcommand{\p}{\partial}

\newcommand{\DG}{\Delta^{\!\Gamma}}
\newcommand{\DP}{\Delta^{\!\partial P}}
\newcommand{\PPk}{\mathcal P_k} 

\newcommand{\bG}{b_{_{\Gamma}}}

\begin{document}

\title{Homology of spectral minimal partitions}

\author[G. Berkolaiko]{Gregory Berkolaiko}
\address{Department of
  Mathematics, Texas A\&M University, College Station, TX 77843-3368, USA}
\email{berko@math.tamu.edu}

\author[Y. Canzani]{Yaiza Canzani}
\address{Department of Mathematics, University of North Carolina at Chapel Hill,
Phillips Hall, Chapel Hill, NC  27599, USA}
\email{canzani@email.unc.edu}

\author[G. Cox]{Graham Cox}
\address{Department of Mathematics and Statistics, Memorial University of Newfoundland, St. John's, NL A1C 5S7, Canada}
\email{gcox@mun.ca}

\author[J.L. Marzuola]{Jeremy L. Marzuola}
\address{Department of Mathematics, University of North Carolina at Chapel Hill,
Phillips Hall, Chapel Hill, NC  27599, USA}
\email{marzuola@math.unc.edu}

\begin{abstract}
A spectral minimal partition of a manifold is its decomposition into disjoint open sets that minimizes a spectral energy functional. It is known that bipartite spectral minimal partitions coincide with nodal partitions of Courant-sharp Laplacian eigenfunctions.  However, almost all minimal partitions are non-bipartite.  To study those, we define a modified Laplacian operator and prove that the nodal partitions of its Courant-sharp eigenfunctions are minimal within a certain topological class of partitions. This yields new results in the non-bipartite case and recovers the above known result in the bipartite case. Our approach is based on tools from algebraic topology, which we illustrate by a number of examples where the topological types of partitions are characterized by relative homology.
\end{abstract}

\maketitle

\section{Introduction}

Let $M$ be a compact, oriented surface, with or without boundary, endowed with a Riemannian metric. A \textit{$k$-partition of $M$} is a mutually disjoint collection  $P = \{\Omega_i\}_{i=1}^k$ of nonempty, open, connected subsets of $M$. Letting $\lambda_1(\Omega_i)$ denote the smallest eigenvalue of the Dirichlet Laplacian on $\Omega_i$, we define the \textit{energy}
\begin{equation}
\label{Lambda}
	\Lambda(P) = \max_{1\leq i\leq k} \lambda_1(\Omega_i)
\end{equation}
and say that $P$ is \textit{minimal} if $\Lambda(P) \leq \Lambda(\tilde P)$ for every $k$-partition $\tilde P$.

Minimal partitions arise in a wide variety of applications. They first appeared in the study of free boundary variational problems \cite{ACF} and spatially segregated solutions to reaction--diffusion systems \cite{CTV2005,CTV2005s}. 
In \cite{royo2014segregation} minimal partitions were considered in relation to ground states of multi-component Bose--Einstein condensates, and in \cite{miclo2015hyperboundedness} a version of the minimal spectral partition problem was used to prove the existence of spectral gaps for ergodic reversible Markov operators.  In \cite{osting2017consistency}, spectral minimal partitions on manifolds have been shown to arise as a high density limit of related partitioning schemes proposed on graphs \cite{osting2014minimal,zosso2015dirichlet,wang2019diffusion}.

It is known that minimal $k$-partitions exist for each $k$ and are suitably regular; see \cite{HHOT} and references therein. However, rigorously finding minimal partitions is quite difficult in general. An important result is the following, where we recall that a partition is \textit{bipartite} if each of the domains $\Omega_i$ can be assigned one of two colors in such a way that neighboring domains have different colors ($\Omega_i$ and $\Omega_j$ are \emph{neighbors} if $\mbox{int} (\overline{\Omega_i \cup \Omega_j}) \neq \Omega_i \cup \Omega_j$).

\begin{theorem}[Helffer, Hoffmann-Ostenhof, Terracini \cite{HHOT}]
\label{thm:bi}
A bipartite $k$-partition is minimal if and only if it is the nodal partition of a $k$-th (i.e.\ Courant sharp) eigenfunction of the Laplacian.
\end{theorem}

See \eqref{SPosition} for the precise definition of Courant sharp. A given Riemannian manifold can only have a finite number of Courant sharp eigenfunctions \cite{P56}, so it follows that most minimal partitions are not bipartite, but rigorously identifying them is a difficult problem.  For instance, while it is known that the minimal 3-partitions of the square and the disk cannot be bipartite, and the minimizing partitions have been conjectured, they have not been rigorously verified except in the presence of a priori assumptions on their structure. See \cite{bonnaillie2015nodal} for a relatively up-to-date summary of the literature.

Our main result, \Cref{thm:main}, extends \Cref{thm:bi} to the non-bipartite case by establishing minimality within a certain class of related partitions. This class is defined using a concept from algebraic topology, namely the homology class of the \emph{boundary set} $\p P = \cup_i \overline{\pO_i\backslash\p M}$ of a partition $P = \{\Omega_i\}$.  We thus develop a new tool for studying non-bipartite partitions and also provide further insight into the bipartite case, as bipartite partitions are precisely those partitions whose boundary sets are null homologous.

To describe the partitions under consideration, recall that a $C^1$ curve is said to be \emph{regular} if its velocity never vanishes.

\begin{define}
\label{def:regSigma}
Let $M$ be a compact surface with piecewise $C^1$ (or empty) boundary. A closed set $\Gamma \subset M$ is a \emph{piecewise $C^1$ cut} if it is the image of a finite set of regular $C^1$ curves that intersect one other (and $\p M$) transversely, and only do so at their endpoints.
\end{define}

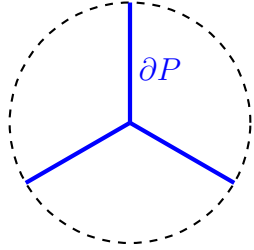
\begin{figure}
\begin{center}
\begin{tikzpicture}[ scale=0.8]
	\draw[thick,dashed] (2,0) arc[radius=2, start angle=0, end angle=360];
	\draw[ultra thick, blue] (0,0) -- (0,2);
	\draw[ultra thick, blue] (0,0) -- ({2*cos(210)},{2*sin(210)});
	\draw[ultra thick, blue] (0,0) -- ({2*cos(330)},{2*sin(330)});
	\node[blue] at (0.5,0.9) {$\p P$};
\end{tikzpicture}
\end{center}
\caption{The conjectured (but unproven) minimal 3-partition of the disk.}
\label{fig:Mercedes}
\end{figure}

\begin{rem}
The $C^1$ assumption simplifies exposition and is justified by the fact that the boundary set of a minimal partition is a piecewise $C^1$ cut \cite{HHOT}.  However, the constructions and results described here extend to cuts $\Gamma$ that are merely continuous, see \Cref{sec:cts}.
\end{rem}

We also need non-bipartite versions of ``nodal partitions" and ``Courant sharp eigenfunctions" that were introduced in \cite{BCHS}. Given a piecewise $C^1$ cut $\Gamma$, we define an operator $-\DG$ acting as the (positive-definite) Laplacian on the space of functions on $M$ satisfying anti-continuity conditions across $\Gamma$. The precise definition, which we give in \Cref{sec:Plap}, is complicated by the fact that \Cref{def:regSigma} allows $\Gamma$ to have cracks. In the special case that each component $\Omega_i$ of the corresponding partition is a Lipschitz domain,
the restrictions $u_i = u|_{_{\Omega_i}}$ of any function $u$ in the domain of $-\DG$ are required to satisfy\footnote{The second equation is also an anti-continuity condition, since $\nu_i = - \nu_j$ on the common boundary.}
\begin{equation}
\label{DPboundary}
	u_i\big|_{\pO_i \cap \pO_j} = - u_j\big|_{\pO_i \cap \pO_j}, \qquad \frac{\p u_i}{\p \nu_i}\bigg|_{\pO_i \cap \pO_j} = \frac{\p u_j}{\p \nu_j}\bigg|_{\pO_i \cap \pO_j}
\end{equation}
whenever $\Omega_i$ and $\Omega_j$ are neighbors; $\nu_i$ denotes the outward unit normal to $\Omega_i$, and we impose Dirichlet boundary conditions on $\p M$ if it is nonempty.

Let $\{\lambda_n(\Gamma)\}_{n=1}^\infty$ denote the eigenvalues of $-\DG$, listed in increasing order and repeated according to their multiplicity. If $\psi$ is an eigenfunction with eigenvalue $\lambda_\psi$, we define its \textit{spectral position}  
\begin{equation}\label{SPosition}
	\ell(\psi) = \min\{n : \lambda_n(\Gamma) = \lambda_\psi\}.
\end{equation}
Defining the \emph{nodal domains} of $\psi$ to be the connected components of $\{x \in M : \psi(x) \neq 0\}$, it follows from \cite[Thm.~1.7]{BCHS} that $\psi$ has at most $\ell(\psi)$ nodal domains. We say that $\psi$ is \emph{Courant sharp} if it has exactly $\ell(\psi)$ nodal domains. We also define the \emph{nodal partition} $P_\psi$ of $\psi$ to be the partition of $M$ whose components are the nodal domains of $\psi$.

\begin{rem}\label{UEquivalent}
If $P$ is bipartite, then $-\DP$ is unitarily equivalent to the Laplacian and \cite[Thm.~1.7]{BCHS} reduces to Courant's nodal domain theorem \cite[\S VI.6]{CH53}; see \Cref{rem:unitary} for details.
\end{rem}

Our main result is that the nodal partition of a Courant-sharp eigenfunction of $-\DG$ minimizes energy within a certain topological class of partitions, which we now specify. The following definition requires some notation and terminology from algebraic topology, which will be developed in \Cref{sec:hom}.

\begin{define}
\label{def:hom}
Piecewise $C^1$ cuts $\Gamma_1$ and $\Gamma_2$ are \emph{homologous} if they are represented by singular 1-chains $\gamma_1$ and $\gamma_2$ for which $[\gamma_1 - \gamma_2] = 0 \in H_1(M, \p M; \bbZ_2)$.
\end{define}

A more intuitive formulation is the following: If $\Gamma_1 \cup \Gamma_2$ is also a piecewise $C^1$ cut, then $\Gamma_1$ is homologous to $\Gamma_2$ if and only if the closure of the symmetric difference, 
\[
\overline{\Gamma_1 \triangle \Gamma_2} = \overline{(\Gamma_1 \cup \Gamma_2) \setminus (\Gamma_1 \cap \Gamma_2)},
\]
is the boundary set of a bipartite partition; see \Cref{fig:diskhom} for an illustration and \Cref{thm:homdef} for a precise statement. In particular, letting $\Gamma_2 = \varnothing$, we see that $\Gamma_1$ is \emph{null homologous} if and only if it is the boundary set of a bipartite partition.

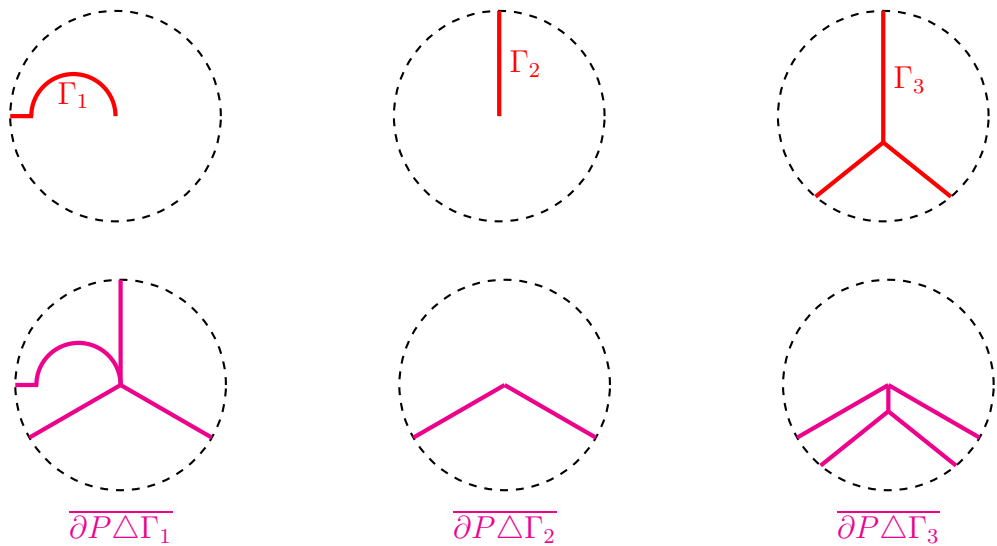
\begin{figure}
\begin{tikzpicture}[ scale=0.7]
	\draw[thick,dashed] (2,0) arc[radius=2, start angle=0, end angle=360];
	\draw[ultra thick, red] (0,0) arc[radius=0.8, start angle=0, end angle=180];
	\draw[ultra thick, red] (-1.565,0) -- (-2,0);
	\node[red] at (-0.8,0.4) {$\Gamma_1$};
\end{tikzpicture}
\hspace{2cm}
\begin{tikzpicture}[ scale=0.7]
	\draw[thick,dashed] (2,0) arc[radius=2, start angle=0, end angle=360];
	\draw[ultra thick, red] (0,0) -- (0,2);
	\node[red] at (0.5,1) {$\Gamma_2$};
\end{tikzpicture}
\hspace{2cm}
\begin{tikzpicture}[ scale=0.7]
	\draw[thick,dashed] (2,0) arc[radius=2, start angle=0, end angle=360];
	\draw[ultra thick, red] (0,-0.5) -- (0,2);
	\draw[ultra thick, red] (0,-0.5) -- ({2*cos(230)},{2*sin(230)});
	\draw[ultra thick, red] (0,-0.5) -- ({2*cos(310)},{2*sin(310)});
	\node[red] at (0.5,0.7) {$\Gamma_3$};
\end{tikzpicture}
\vspace{0.7cm}

\begin{tikzpicture}[ scale=0.7]
	\draw[thick,dashed] (2,0) arc[radius=2, start angle=0, end angle=360];
	\draw[ultra thick, magenta] (0,0) arc[radius=0.8, start angle=0, end angle=180];
	\draw[ultra thick, magenta] (-1.565,0) -- (-2,0);
	\draw[ultra thick, magenta] (0,0) -- (0,2);
	\draw[ultra thick, magenta] (0,0) -- ({2*cos(210)},{2*sin(210)});
	\draw[ultra thick, magenta] (0,0) -- ({2*cos(330)},{2*sin(330)});
	\node[magenta] at (0,-2.7) {$\overline{\p P\triangle\Gamma_1}$};
\end{tikzpicture}
\hspace{2cm}
\begin{tikzpicture}[ scale=0.7]
	\draw[thick,dashed] (2,0) arc[radius=2, start angle=0, end angle=360];
	\draw[ultra thick, magenta] (0,0) -- ({2*cos(210)},{2*sin(210)});
	\draw[ultra thick, magenta] (0,0) -- ({2*cos(330)},{2*sin(330)});
	\node[magenta] at (0,-2.7) {$\overline{\p P \triangle\Gamma_2}$};
\end{tikzpicture}
\hspace{2cm}
\begin{tikzpicture}[ scale=0.7]
	\draw[thick,dashed] (2,0) arc[radius=2, start angle=0, end angle=360];
	\draw[ultra thick, magenta] (0,0) -- ({2*cos(210)},{2*sin(210)});
	\draw[ultra thick, magenta] (0,0) -- ({2*cos(330)},{2*sin(330)});
	\draw[ultra thick, magenta] (0,-0.5) -- (0,0);
	\draw[ultra thick, magenta] (0,-0.5) -- ({2*cos(230)},{2*sin(230)});
	\draw[ultra thick, magenta] (0,-0.5) -- ({2*cos(310)},{2*sin(310)});
	\node[magenta] at (0,-2.7) {$\overline{\p P\triangle\Gamma_3}$};
\end{tikzpicture}
\caption{The boundary set of the Mercedes star partition, \Cref{fig:Mercedes}, is homologous to $\Gamma_1$ and $\Gamma_2$ but not $\Gamma_3$.}
\label{fig:diskhom}
\end{figure}

Let $\PPk$ denote the set of all $k$-partitions of $M$. For any piecewise $C^1$ cut $\Gamma$ and positive integer $k$ we define 
\begin{equation}
\label{Pkdef}
	\cP_k(\Gamma) = \big\{ \tilde P \in \PPk : \p \tilde P \text{ contains a piecewise $C^1$ cut homologous to } \Gamma \big\}.
\end{equation}
If $\tilde P \in \cP_k(\Gamma)$, then $\p\tilde P$ contains a piecewise $C^1$ cut homologous to $\Gamma$, say $\tilde\Gamma$, but no regularity is assumed for the rest of the boundary set. For instance, if $\Gamma$ is null homologous, then we can choose $\tilde\Gamma = \varnothing$ to conclude that $\cP_k(\Gamma)$ contains \emph{all} $k$-partitions.  We also note that if an eigenfunction $\psi$ of $-\Delta^\Gamma$ has $k$ nodal domains and $P_\psi$ is its nodal partition, then $P_\psi \in \cP_k(\Gamma)$ (see \Cref{lem:equality} for further details).

Our main result compares the $k$th eigenvalue, $\lambda_k(\Gamma)$, of $-\DG$ to the energy of any $\tilde P \in \cP_k(\Gamma)$. The following is a special case of the more general \Cref{thm:cts}, which we will formulate and prove in \Cref{sec:cts}.

\begin{theorem}
\label{thm:main}
If $\Gamma$ is a piecewise $C^1$ cut and $k\in \mathbb N$, then
\begin{equation}
	\lambda_k(\Gamma) \leq \inf \big\{\Lambda(\tilde P):\; \tilde P \in \cP_k(\Gamma) \big\},
\end{equation}
with $\lambda_k(\Gamma) = \Lambda(\tilde P)$ if and only if $\tilde P$ is the nodal partition of an eigenfunction of $-\DG$ with eigenvalue $\lambda_k(\Gamma)$.

In particular, if $\psi$ is a Courant-sharp eigenfunction of $-\DG$ with $k$ nodal domains (and therefore $\ell(\psi)=k$), then $\Lambda(P_\psi) \leq \Lambda(\tilde P)$ 
for all $\tilde P \in \cP_k(\Gamma)$.
\end{theorem}

That is, the nodal partition $P$ of a Courant-sharp eigenfunction for $-\DG$ minimizes energy within the class of partitions $\cP_k(\Gamma)$. When $P$ is bipartite the set $\cP_k(\Gamma)$ consists of \emph{all} $k$-partitions, and nodal partitions of $-\DG$ eigenfunctions are nodal partitions of Laplacian eigenfunctions, so we recover one direction of \Cref{thm:bi} as a special case.

If $\lambda_k(\Gamma)$ is a simple eigenvalue of $-\DG$, we have $\Lambda(P) < \Lambda(\tilde P)$ 
for all $\tilde P \in \cP_k(\Gamma) \backslash \{P\}$. In general, the characterization of equality in \Cref{thm:main} is difficult to apply, as it requires knowing all of the eigenvalues and eigenfunctions of $-\DG$. Our next result gives a sufficient condition for inequality that is topological in nature.

\begin{cor}
\label{cor:strict}
Let $P \in \PPk$ be the nodal partition of a Courant-sharp eigenfunction of $-\DG$. If the boundary set of $\tilde P \in \cP_k(\Gamma)$ is a piecewise $C^1$ cut that is not homologous to $\Gamma$, then $\Lambda(P) < \Lambda(\tilde P)$. In particular, if $P \in \PPk$ is the nodal partition of a Courant-sharp Laplacian eigenfunction, then it has strictly lower energy than any non-bipartite $k$-partition.
\end{cor}

In Section \ref{sec:app} we apply \Cref{thm:main} and \Cref{cor:strict} to partitions of the disk, sphere, torus and cylinder. For now, we mention just one example, to give an idea of the kind of results we obtain.

\begin{theorem}
\label{cor:disk}
The Mercedes star partition is minimal among all 3-partitions whose boundary set contains a curve from the origin to the boundary of the disk. Moreover, it has strictly lower energy than any 3-partition whose boundary set contains such a curve but is not homologous to it.
\end{theorem}

For instance, the set $\Gamma_3$ shown in \Cref{fig:diskhom} contains a curve from the origin to the boundary of the disk, but $\Gamma_3$ is not homologous to this curve, therefore the corresponding 3-partition has strictly greater energy than the Mercedes star partition. (Equivalently, $\Gamma_3$ is not homologous to the boundary set of the Mercedes star partition but contains a subset that is, namely the set $\Gamma_2$ shown in \Cref{fig:diskhom}.)

\begin{rem}
\label{rem:disk}
Theorem \ref{cor:disk} should be compared to the results for the disk in \cite{BHdisk,HHOdisk}, where it is shown that \emph{if} the boundary set of the \emph{minimal} 3-partition has only one singular point, or contains the origin, then it must be the Mercedes star. While our topological hypothesis is more restrictive, we make no a priori assumptions about the structure of the minimizer. That is, we show that the Mercedes star minimizes energy within a certain topological class of partitions, whether or not that class contains a minimal partition. As a byproduct, we deduce the existence of a minimizer \emph{in this class}.
\end{rem}

The remaining applications have a similar flavor\,---\,starting from Courant-sharp eigenfunctions for a given $-\Delta^\Gamma$, we construct partitions that minimize energy within a certain topological class. This does not require us to make any a priori assumptions on the structure of the minimal $k$-partition, nor does it require us to know that the constrained optimization problem (minimizing energy within a given topological class) has a minimum.

Finally, we remark that while \Cref{thm:main} can likely be reformulated and proved using the ``double covering'' construction\footnote{Here one introduces a suitable branched double cover of $M$ and considers the Laplacian restricted to anti-symmetric functions on the resulting singular manifold.} of \cite{helffer1999nodal},  
the homological approach developed in this paper has several advantages: (i) the tools of algebraic topology can be brought to bear on the problem of minimal partitions; (ii) the topological hypotheses of our theorems are formulated on the original manifold and do not involve lifting to the double cover (which is easy to construct and visualize in concrete examples but rather abstract in general); (iii) the partition Laplacian $-\Delta^\Gamma$ is easier to analyze, visualize and prepare for numerical computations, since it is defined on the smooth manifold $M$, rather than its branched double cover.

\subsection*{Outline}

In \Cref{sec:Plap} we define the operator $-\DG$ and state its relevant properties (to be proved later), which we then use to prove \Cref{thm:main}. In \Cref{sec:hom} we review some topological concepts and use them to clarify \Cref{def:hom}. In \Cref{sec:unitary} we prove the technical results from \Cref{sec:Plap} and hence complete the proof of \Cref{thm:main}. \Cref{cor:strict} is proved in \Cref{Eq}. In \Cref{sec:cts} we generalize all of our results to partitions with continuous, rather than piecewise $C^1$, boundary sets. Finally, in \Cref{sec:app} we apply our results to partitions of the disk, sphere, torus and cylinder.


\section{The partition Laplacian and proof of Theorem \ref{thm:main}}
\label{sec:Plap}

We define $-\DG$ in \Cref{sec:def} and then characterize its dependence on $\Gamma$ in \Cref{prop:DG}, which will be proved in \Cref{sec:unitary}. In \Cref{sec:proof} we use \Cref{prop:DG} to prove \Cref{thm:main}.

\subsection{The partition Laplacian}
\label{sec:def}
Let $\Gamma$ be a piecewise $C^1$ cut. For convenience we write $M^o = M \backslash \p M$. To define $-\DG$, we need to define \emph{traces} (restrictions to $\Gamma$) for functions in the Sobolev space $W^{1,2}(M^o\backslash\Gamma)$. 
\Cref{def:regSigma} allows a subdomain to lie on both sides of its boundary, as in \Cref{fig:trace}, so a given $\Omega_i$ may not be a Lipschitz domain.

\begin{figure}
\begin{tikzpicture}[ scale=0.9]
	\draw[thick,dashed] (-2,-2) -- (-2,2) -- (2,2) -- (2,-2) -- (-2,-2);
	\draw[very thick] (0,0) -- (0,2);
	\node at (0,-0.75) {$\Gamma$};
\end{tikzpicture}
\hspace{1cm}
\begin{tikzpicture}[ scale=0.9]
	\draw[thick,dashed] (-2,-2) -- (-2,2) -- (2,2) -- (2,-2) -- (-2,-2);
	\draw[very thick, blue] (0,0) -- (0,2);
	\draw[very thick, blue] (0,1.98) arc[radius=1.13, start angle=120, end angle=240];
	\node[blue] at (0,-0.75) {$\tilde\Gamma$};
\end{tikzpicture}

\caption{The slit in $\Gamma$ has been eliminated in the set $\tilde\Gamma$ by adding an extra segment, so each connected component of $M^o \backslash \tilde\Gamma$ is a Lipschitz domain.}
\label{fig:trace}
\end{figure}
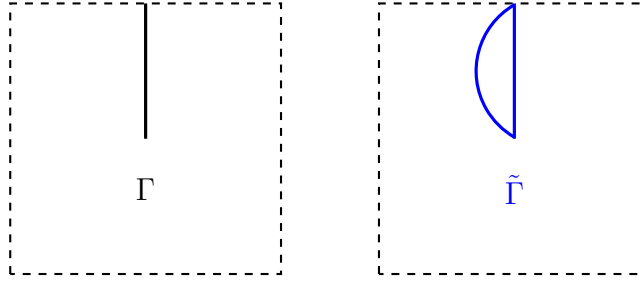

We therefore extend $\Gamma$ to a piecewise $C^1$ cut $\tilde\Gamma$ such that each component of $M^o \backslash \tilde\Gamma$ lies on one side of its boundary, and hence is a Lipschitz domain. See \cite[\S 1.7]{Gr}), and also \cite{BCE}, where similar techniques were used to study Neumann domains.

Label the components of $M^o \backslash \tilde\Gamma$ by $\{\tilde\Omega_i\}$. A regular segment $\Gamma_a \subset\Gamma$ will be contained in $\p\tilde\Omega_i \cap \p\tilde\Omega_j$ for unique $i \neq j$. Since the subdomains $\tilde\Omega_i$ and $\tilde\Omega_j$ are Lipschitz, for any function $u \in W^{1,2}(M^o \backslash \Gamma)$ we can define two traces,
\begin{equation}
\label{twotraces}
	u_i\big|_{\Gamma_a}, \ u_j\big|_{\Gamma_a} \in W^{1/2, 2}(\Gamma_a),
\end{equation}
where $u_i = u|_{_{\tilde\Omega_i}}$. These traces do not depend on the choice of extension $\tilde\Gamma$. Similarly, since $\p M$ is piecewise $C^1$, it can be decomposed into a finite set of regular $C^1$ curves, each contained in the boundary of a unique $\tilde\Omega_i$. For every such $\tilde\Omega_i$, which we call a \emph{boundary subdomain}, the restriction $u_i|_{_{\p M \cap \p \tilde\Omega_i}}$ is thus well defined.

\begin{define}
$u \in W^{1,2}(M^o \backslash \Gamma)$ is \emph{continuous across $\Gamma_a\subset \Gamma$} if $u_i|_{_{\Gamma_a}} = u_j|_{_{\Gamma_a}}$, and is \emph{anti-continuous across $\Gamma_a$} if $u_i|_{_{\Gamma_a}} = -u_j|_{_{\Gamma_a}}$. It is \emph{(anti)continuous} across $\Gamma$ if it is (anti)continuous across each $\Gamma_a$. It \emph{vanishes on $\p M$} if $u_i|_{_{\p M \cap \p \tilde\Omega_i}} = 0$ for every boundary subdomain $\tilde\Omega_i$.
\end{define}

We now define
\begin{equation}\label{H1def}
	W^{1,2}_0(M; \Gamma) =  \Big\{ u \in W^{1,2}(M^o\backslash\Gamma) : u \text{ vanishes on $\p M$ and is anti-continuous across $\Gamma$} \Big\}.
\end{equation}
This is a closed subspace of $W^{1,2}(M^o \backslash\Gamma)$ and hence is complete. When $\Gamma$ is empty it coincides with $W^{1,2}_0(M^o)$.

Letting $dV$ denote the Riemannian volume form on $M$, $\cdot$ the inner product, and setting $u_i = u|_{_{\Omega_i}}$, we define the bilinear form
\begin{equation}
\label{bGamma}
	\bG(u,v) = \sum_{i=1}^k \int_{\Omega_i} (\nabla u_i \cdot \nabla v_i)\,dV, \qquad  \dom(\bG) = W^{1,2}_0(M;\Gamma),
\end{equation}
where $\{\Omega_i\}$ are the connected components of $M^o\backslash\Gamma$.
The form $\bG$ is nonnegative and hence lower semibounded. It is a closed form, because $W^{1,2}_0(M;\Gamma)$ is a closed subspace of $W^{1,2}(M^o \backslash\Gamma)$ and hence is complete. It thus generates a self-adjoint operator on $L^2(M)$, which we denote by $-\DG$, via the representation theorem for semibounded forms; see, for instance \cite[Thm~10.7]{Schm12}.  Note that $\Gamma = \varnothing$ corresponds to the Dirichlet Laplacian on $M$.

Since each $\Omega_i$ is a Lipschitz domain, we have that $W^{1,2}(\Omega_i)$ is compactly embedded in $L^2(\Omega_i)$, from which it follows that $W^{1,2}_0(M;\Gamma)$ is compactly embedded in $L^2(M)$. As a result,  $-\DG$ has compact resolvent, so its spectrum consists of isolated eigenvalues of finite multiplicity. Writing the eigenfunctions and eigenvalues of $-\DG$ as $\big(\phi_n(\Gamma), \lambda_n(\Gamma)\big)$, we have the variational characterization
\begin{equation}
\label{Ray}
	\lambda_n(\Gamma) = \min \left\{ \frac{\bG(\phi,\phi)}{\int_M \phi^2\,dV} \, : \, \phi \in \{\phi_1(\Gamma), \ldots, \phi_{n-1}(\Gamma)\}^\perp \cap H^1_0(M;\Gamma) \backslash \{0\} \right\}.
\end{equation}
The minimum is achieved if and only if $\phi$ is an eigenfunction for the eigenvalue $\lambda_n(\Gamma)$.

Having defined $-\DG$, we now describe its dependence on the cut $\Gamma$.

\begin{prop}
\label{prop:DG}
If $\Gamma_1$ and $\Gamma_2$ are homologous, then $-\Delta^{\!\Gamma_1}$ and $-\Delta^{\!\Gamma_2}$ are unitarily equivalent. That is, there is a unitary operator $\Phi$ on $L^2(M)$ that maps $W^{1,2}_0(M;\Gamma_1)$ onto $W^{1,2}_0(M; \Gamma_2)$ and satisfies
\begin{equation}
	b_{_{\Gamma_2}}(\Phi u, \Phi v) = b_{_{\Gamma_1}}(u,v)
\end{equation}
for all $u,v \in W^{1,2}_0(M; \Gamma_1)$.  Moreover, the unitary transformation $\Phi$ preserves nodal sets of eigenfunctions, i.e., $\{ u = 0 \} = \{ \Phi u = 0 \}$ if $-\Delta^{\!\Gamma_1} u = \lambda u$ for some $\lambda \in \bbR$.
\end{prop}

This is the most technical result of the paper. The proof is somewhat delicate because $\Gamma_1$ and $\Gamma_2$ need not be related in a nice way. While it is assumed that the curves in $\Gamma_1$ intersect one another transversely, and likewise for the curves in $\Gamma_2$, it is possible that some curve in $\Gamma_1$ intersects a curve in $\Gamma_2$ tangentially, in which case $\Gamma_1 \cup \Gamma_2$ is not a piecewise $C^1$ cut. This makes it difficult to compare the domains of the operators $-\Delta^{\!\Gamma_1}$ and $-\Delta^{\!\Gamma_2}$. We will resolve this issue in \Cref{sec:unitary} using a transversality argument.

\begin{rem}
\label{rem:unitary}
If $P$ is bipartite, then $-\DP$ is unitarily equivalent to the Laplacian and we get half of \Cref{thm:bi} as a special case of \Cref{thm:main}.
\end{rem}

\subsection{Proof of main theorem}
\label{sec:proof}

We now show how \Cref{prop:DG} implies \Cref{thm:main}.  The proof is similar to that of \cite[Prop.~5.5]{HHOT}, which covers the bipartite case.

Let $\Gamma$ be a piecewise $C^1$ cut and $k\in \mathbb N$.
 Starting with a partition $\tilde P \in \cP_k(\Gamma)$, we use its groundstates to build a test function for the Rayleigh quotient in \eqref{Ray} and hence obtain an upper bound on $\lambda_k(\Gamma)$. The difficulty in the non-bipartite case is producing a test function in $W^{1,2}_0(M; \Gamma)$, since the groundstates coming from $\tilde P$ are continuous, rather than anti-continuous, across $\Gamma$. We resolve this by modifying the groundstates on regions bounded by $\Gamma$ and $\p \tilde P$. This is where the hypothesis $\tilde P \in \cP_k(\Gamma)$ is used.

\begin{proof}[Proof of \Cref{thm:main}]

Let $\tilde P = \{ \tilde \Omega_i\}$ be an arbitrary partition in $\cP_k(\Gamma)$, with $\tilde\psi_i$ denoting the groundstate for the Dirichlet Laplacian on $\tilde\Omega_i$, extended by zero to the rest of $M$ and normalized so that
\begin{equation}
	\int_{\tilde\Omega_i} |\nabla \tilde\psi_i|^2\,dV = \lambda_1(\tilde\Omega_i).
\end{equation}
Since $\tilde P \in \cP_k(\Gamma)$, there exists a subset $\tilde\Gamma \subset \p \tilde P$ that is homologous to $\Gamma$. Since $\tilde\psi_i$ vanishes on $\p \tilde P$, it is anti-continuous across $\tilde\Gamma$, therefore $\tilde\psi_i \in W^{1,2}_0(M; \tilde\Gamma)$. Since $\tilde \Gamma$ is homologous to  $\Gamma$, by \Cref{prop:DG} there is a unitary operator $\Phi$ acting on $L^2(M)$ so that $\Phi(W^{1,2}_0(M; \tilde\Gamma))=W^{1,2}_0(M; \Gamma)$. Then, letting $\psi_i = \Phi \tilde\psi_i \in W^{1,2}_0(M; \Gamma)$, we obtain
\begin{equation}
\label{Hat}
	b_{_{\Gamma}}( \psi_i,  \psi_i) = b_{_{\tilde\Gamma}}(\tilde\psi_i, \tilde\psi_i)
	 = \int_{\tilde\Omega_i} |\nabla \tilde\psi_i|^2\,dV = \lambda_1(\tilde\Omega_i).
\end{equation}

Now choose nonzero real numbers $c_1, \ldots, c_k$ so that the linear combination $\phi = c_1 \psi_1 + \cdots + c_k \psi_k$ is orthogonal to the first $k-1$ eigenfunctions of $-\DG$
and hence is a valid test function for \eqref{Ray}. 
Since the $\psi_i$ have mutually disjoint supports, we can use \eqref{Ray} and \eqref{Hat} to compute
\begin{equation}
\label{minmaxproof}
	\lambda_k(P) \leq \frac{b_{_{\Gamma}}(\phi,\phi)}{\int_M \phi^2\,dV} 
	= \frac{c_1^2 \lambda_1(\tilde\Omega_1) + \cdots + c_k^2 \lambda_1(\tilde\Omega_k)}{c_1^2 + \cdots + c_k^2} 
	\leq \max_{1 \leq i \leq k} \lambda_1(\tilde\Omega_i)
	= \Lambda(\tilde P),
\end{equation}
as claimed.

If $\lambda_k(\Gamma) = \Lambda(\tilde P)$, then the inequalities in \eqref{minmaxproof} are all equalities. From the case of equality in \eqref{Ray} we see that $\phi = c_1 \psi_1 + \cdots + c_k \psi_k$ must be an eigenfunction of $-\DG$ for the eigenvalue $\lambda_k(\Gamma)$, so $\tilde P$ is the nodal partition of an eigenfunction of $-\DG$, for the eigenvalue $\lambda_k(\Gamma)$. Conversely, if $\tilde P$ is the nodal partition of an eigenfunction for the eigenvalue $\lambda_k(\Gamma)$, each subdomain has groundstate energy $\lambda_1(\tilde\Omega_i) = \lambda_k(\Gamma)$ and hence $\tilde P$ has energy $\Lambda(\tilde P) = \lambda_k(\Gamma)$. The same argument shows that if $P$ is a $k$-partition generated by a Courant-sharp eigenfunction of $-\DG$, then $\Lambda(P) = \lambda_k(\Gamma)$ and so $\Lambda(P) \leq \Lambda(\tilde P)$ for all $\tilde P \in \cP_k(\Gamma)$.
\end{proof}


\section{Homology of cuts}
\label{sec:hom}

We now elaborate on \Cref{def:hom}. After reviewing some concepts from algebraic topology, we explain how piecewise $C^1$ cuts, $\Gamma_1$ and $\Gamma_2$, can be represented by singular 1-chains, $\gamma_1$ and $\gamma_2$ respectively. This representation is not unique, but we will see that the condition $[\gamma_1 - \gamma_2] = 0 \in H_1(M, \p M; \bbZ_2)$ is independent of this choice, and so \Cref{def:hom} is meaningful. Note that we have $\gamma_1 - \gamma_2 = \gamma_1 + \gamma_2$, as our coefficients lie in $\bbZ_2$. See \cite[Ch.~2]{Hat} for an introduction to the definitions and tools used here.

In \Cref{sec:null} we give an equivalent formulation of \Cref{def:hom} for a sufficiently regular pair of cuts, and in \Cref{sec:odd} we discuss the relationship between homology and the set of odd points for a cut.

\subsection{Definitions}
\label{sec:homdef}
Letting $\Delta^n \subset \mathbb{R}^n$ denote the standard $n$-simplex, we recall that a singular $n$-simplex is a continuous map $\Delta^n \to M$. In particular, a singular 0-simplex corresponds to a point in $M$ and a singular 1-simplex is a parameterized curve. A singular $n$-chain is a (finite) formal sum of singular $n$-simplices with coefficients in $\bbZ_2$, so the set of these, denoted $C_n(M; \bbZ_2)$, is the free $\bbZ_2$-module generated by the singular $n$-simplices.

For each $n$ there exists a boundary map $\p_n \colon C_n(M; \bbZ_2) \to C_{n-1}(M; \bbZ_2)$ with the property that $\p_n \p_{n+1} = 0$. For instance, viewing a continuous curve $\gamma \colon [0,1] \to M$ as an element of $C_1(M; \bbZ_2)$, we have $\p_1\gamma = \gamma(1) - \gamma(0) \in C_0(M; \bbZ_2)$.

Defining the quotient groups $C_n(M,\p M; \bbZ_2) = C_n(M; \bbZ_2)/C_n(\p M; \bbZ_2)$, we see that each $\p_n$ descends to a boundary map $\p_n \colon C_n(M, \p M; \bbZ_2) \to C_{n-1}(M, \p M; \bbZ_2)$ again satisfying $\p_n \p_{n+1} = 0$. The relative homology groups are given by $H_n(M,\p M; \bbZ_2) = \ker(\p_n) / \operatorname{im} (\p_{n+1})$ for each $n$. In particular, $H_1(M,\p M; \bbZ_2)$ consists of equivalence classes of relative 1-cycles, i.e., 1-chains $\gamma \in C_1(M; \bbZ_2)$ for which $\p_1\gamma \in C_0(\p M;\bbZ_2)$.

The homology class of a relative 1-cycle $\gamma$ is denoted by $[\gamma] \in H_1(M,\p M; \bbZ_2)$. Note that $[\gamma] = 0$ if and only if $\gamma = \p_2 \omega + \beta$ for some $\omega \in C_2(M; \bbZ_2)$ and $\beta \in C_1(\p M; \bbZ_2)$. We say that two 1-chains $\gamma_1$ and $\gamma_2$ are \emph{homologous} if $[\gamma_1 - \gamma_2] = 0$. For this to hold it is necessary that $\gamma_1 - \gamma_2$ to be a relative 1-cycle, that is, $\p_1 \gamma_1 - \p_1 \gamma_2 \in C_0(\p M;\bbZ_2)$.  This condition is not sufficient: see the example in \Cref{fig:annulus}, where a relative 1-cycle $\Gamma_2$ is not null-homologous.

\subsection{The homology class of a piecewise $C^1$ cut}
Given a partition, we are interested in the homology class of its boundary set. If $\Gamma$ is a piecewise $C^1$ cut, in the sense of \Cref{def:regSigma}, then it can be viewed as a singular 1-chain. More precisely, given a finite collection $\{\gamma_a\}$ of regular $C^1$ curves that can only intersect at their end points, whose image is $\Gamma$, we define $\gamma = \sum_a \gamma_a \in C_1(M; \bbZ_2)$. This 1-chain is said to \emph{represent $\Gamma$}.

A cut $\Gamma$ can be represented by many different chains. This does not matter for the following reason.

\begin{lemma}
\label{lem:param}
If $\gamma, \tilde\gamma \colon [0,1] \to M$ are continuous, then their concatenation $\gamma \ast \tilde\gamma$ is homologous to $\gamma + \tilde\gamma$. If $\tau \colon [0,1] \to [0,1]$ is a homeomorphism, then $\gamma \circ \tau$ is homologous to $\gamma$.
\end{lemma}

We omit the proof, which is standard, but note that the second statement relies on the fact that we are using $\bbZ_2$ coefficients, and thus have $\gamma = -\gamma$ as 1-chains, since $\tau$ was not assumed to preserve orientation.

It follows from \Cref{lem:param} that all representatives of the same cut are homologous. Therefore, if two cuts $\Gamma_1$ and $\Gamma_2$ can be represented by homologous 1-chains, any other representatives of $\Gamma_1$ and $\Gamma_2$ will also be homologous, and so \Cref{def:hom} is meaningful.

\subsection{Homology and bipartiteness}
\label{sec:null}

In this section we reformulate \Cref{def:hom} for pairs of cuts satisfying a compatibility condition. While more restrictive than the original definition, this formulation is more geometric in nature and suffices for many applications; see \Cref{fig:diskhom} for examples.

\begin{define}
Piecewise $C^1$ cuts $\Gamma_1$ and $\Gamma_2$ are \emph{compatible} if $\Gamma_1 \cup \Gamma_2$ is a piecewise $C^1$ cut.
\end{define}

To understand why such a condition is needed, consider $\gamma_1, \gamma_2 \colon [-1,1] \to [-1,1]^2$ given by
\begin{equation}
	\gamma_1(t) = (t, 0), \qquad \gamma_2(t) = \begin{cases} \big(t,e^{-1/t^2} \sin(1/t) \big), &t \neq 0, \\ (0,0), & t=0. \end{cases}
\end{equation}
These are $C^\infty$ regular curves, thus each of their images is a piecewise $C^1$ cut, but the same is not true of their union, due to the infinitely many intersections of $\gamma_1$ and $\gamma_2$.

We recall that a partition $P = \{\Omega_i\}$ is \emph{bipartite} if each of the $\Omega_i$ can be assigned one of two colors so that neighboring domains have different colors, where  $\Omega_i$ and $\Omega_j$ are \emph{neighbors} if $\mbox{int} (\overline{\Omega_i \cup \Omega_j}) \neq \Omega_i \cup \Omega_j$.  Note that a domain with a slit always neighbors itself, therefore any partition that contains such a domain cannot be bipartite. For instance, the partition with boundary set $\Gamma_2$ in Figure \ref{fig:diskhom} is not bipartite, even though it only has one component.

\begin{theorem}
\label{thm:homdef}
Compatible cuts $\Gamma_1$ and $\Gamma_2$ are homologous if and only if $\overline{\Gamma_1 \triangle \Gamma_2}$ is the boundary set of a bipartite partition.
\end{theorem}

To prove our main results it suffices to know this theorem when $\Gamma_1$ and $\Gamma_2$ intersect transversely; this special case will be used in the proof of \Cref{lem:DG} below. We have chosen to formulate and prove the result in the more general context of compatible cuts as this gives further geometric insight into \Cref{def:hom}, and in particular applies to all of the examples shown in \Cref{fig:diskhom}.

To prove \Cref{thm:homdef}, we start with a technical result about compatible cuts.

\begin{lemma}
\label{lem:compat}
If $\Gamma_1$ and $\Gamma_2$ are compatible, then $\Gamma_1 \cup \Gamma_2$ can be parameterized by regular $C^1$ curves $\{\gamma_a\}_{a \in A}$ such that the image of each $\gamma_a$ is contained in $\overline{\Gamma_1\backslash\Gamma_2}$, $\overline{\Gamma_2\backslash\Gamma_1}$ or $\Gamma_1 \cap \Gamma_2$.  
\end{lemma}

\begin{proof}
By assumption, $\Gamma_1 \cup \Gamma_2$ is the union of the images of a finite set of curves that only intersect at their endpoints. Subdividing if necessary, we can assume that none of these curves contain an endpoint of either $\Gamma_1$ or $\Gamma_2$ in its interior. Let $\gamma \colon [0,1] \to M$ denote one such curve.  If the interior $\gamma\big((0,1)\big)$ is contained in $\Gamma_1\backslash\Gamma_2$ or $\Gamma_2\backslash\Gamma_1$, then there is nothing to prove. Therefore, we just need to show that if $\gamma\big((0,1)\big)$ intersects $\Gamma_1$ and $\Gamma_2$, then it is contained in $\Gamma_1 \cap \Gamma_2$.

We claim that if $\gamma\big((0,1)\big)$ intersects $\Gamma_1$, then $\gamma\big((0,1)\big) \subset \Gamma_1$ (and likewise for $\Gamma_2$). Equivalently, the set $I_1 = \{t \in (0,1) : \gamma(t) \in \Gamma_1\}$ is equal to $(0,1)$. To see this, it suffices to show that this set is relatively open and closed in $(0,1)$; the latter property follows immediately from continuity of $\gamma_1$. 

To show that $I_1$ is relatively open, let $t_* \in I_1$. This means there is a curve $\gamma_1 \colon [0,1] \to M$ in $\Gamma_1$ and $t_1 \in [0,1]$ for which $\gamma_1(t_1) = \gamma(t_*)$. Since the interior of $\gamma$ does not contain an endpoint of any curve in $\Gamma_1$, we must have $t_1 \in (0,1)$. Moreover, we claim that there exists $\delta>0$ such that
\begin{equation}\label{delta}
\gamma_1\big([t_1 - \delta, t_1 + \delta]\big) \subset \gamma\big((0,1)\big).
\end{equation}
Suppose the claim does not hold. Then, for all $\delta$, the restriction of $\gamma_1$ to $[t_1 - \delta, t_1 + \delta]$ will intersect some curve in $\Gamma_1 \cup \Gamma_2$ other than $\gamma$, say $\hat\gamma$. Since for $\delta$ small enough $\gamma_1\big([t_1 - \delta, t_1 + \delta]\big)$ does not intersect the endpoints of $\gamma$, then $\hat\gamma$ must intersect the interior of $\gamma$, yielding a contradiction.

Since $\gamma_1$ is a regular curve and \eqref{delta} holds, it follows from the implicit function theorem that $\gamma_1\big((t_1 - \delta, t_1 + \delta)\big) \subset \gamma\big((0,1)\big)$ is an open neighborhood of $\gamma(t_*)$, and hence contains $\gamma(t)$ for all $t$ in some neighborhood of $t_*$. This proves that $I_1$ is open, as claimed.
\end{proof}

We next give a homological interpretation of the symmetric difference.

\begin{prop}
\label{symdiff}
Compatible cuts $\Gamma_1$ and $\Gamma_2$ are homologous if and only if $\overline{\Gamma_1 \triangle \Gamma_2}$ is null homologous.
\end{prop}

\begin{proof}
Parameterize $\Gamma_1 \cup \Gamma_2$ by curves $\{\gamma_a\}_{a \in A}$ as in \Cref{lem:compat}. Defining $A_1 = \{a \in A : \Im \gamma_a \subset \Gamma_1\} $ and likewise for $A_2$, we have $A = A_1 \cup A_2$ and $A_1 \cap A_2 =  \{a \in A : \Im \gamma_a \subset \Gamma_1 \cap \Gamma_2\} $. It follows that $\gamma_1 = \sum_{a \in A_1} \gamma_a$ represents $\Gamma_1$ and similarly for $\gamma_2 = \sum_{a \in A_2} \gamma_a$.  Since our coefficients are in $\bbZ_2$, we have
\begin{equation}
	\gamma_1 - \gamma_2 = \gamma_1 + \gamma_2 = \sum_{a \in A_1} \gamma_a + \sum_{a \in A_2} \gamma_a 
	= \sum_{a \in  A_1 \triangle A_2} \gamma_a,
\end{equation}
because each $\gamma_1$ with $a \in A_1 \cap A_2$ appears in the sum exactly twice and hence cancels. The chain on the right-hand side represents $\overline{\Gamma_1 \triangle \Gamma_2}$, so we conclude that $[\gamma_1 - \gamma_2] = 0$ if and only if $\overline{\Gamma_1 \triangle \Gamma_2}$ is null homologous.
\end{proof}

\Cref{thm:homdef} is then an immediate consequence of the following.

\begin{prop}
\label{lem:sep}
A piecewise $C^1$ cut $\Gamma \subset M$ is null homologous if and only if it is the boundary set of a bipartite partition.
\end{prop}

\begin{figure}
\begin{tikzpicture}[ scale=0.9]
	\draw[thick,dashed] (-2,-2) -- (-2,2) -- (2,2) -- (2,-2) -- (-2,-2);
	\draw[very thick] (0,-2) -- (0,2);
	\node at (0,-2.5) {$\Gamma$};
	\node at (-1,0) {$\Omega_1$};
	\node at (1,0) {$\Omega_2$};
\end{tikzpicture}
\hspace{1cm}
\begin{tikzpicture}[ scale=0.9]
	\draw[thick,dashed] (-2,-2) -- (-2,2) -- (2,2) -- (2,-2) -- (-2,-2);
	\draw[very thick] (0,-2) -- (0,2);
	\node at (0,-2.5) {$\Gamma$};
	\draw[thick, blue] (-2,-2) -- (0,1) -- (-2,2);
	\draw[thick, blue] (2,-2) -- (0,0) -- (2,2);
\end{tikzpicture}
\hspace{1cm}
\begin{tikzpicture}[ scale=0.9]
	\draw[thick,dashed] (-2,-2) -- (-2,2) -- (2,2) -- (2,-2) -- (-2,-2);
	\draw[very thick] (0,-2) -- (0,2);
	\node at (0,-2.5) {$\Gamma$};
	\draw[thick, blue] (-2,-2) -- (0,1) -- (-2,2);
	\draw[thick, blue] (2,-2) -- (0,0) -- (2,2);
	\draw[thick, red] (-2,-2) -- (0,0);
	\draw[thick, red] (0,1) -- (2,2);
\end{tikzpicture}

\caption{Triangulating $\overline{\Omega}_1$ and $\overline{\Omega}_2$ does not necessarily give a valid triangulation of $M$ (center), but this can always be refined to produce a triangulation (right), as in the proof of \Cref{lem:sep}.
}
\label{fig:tri}
\end{figure}
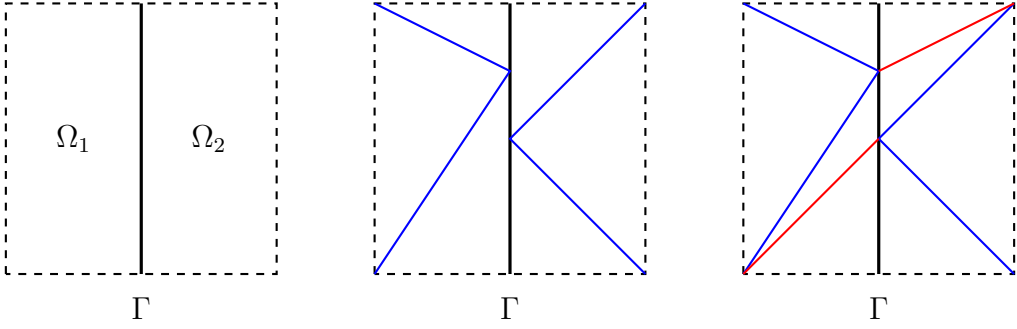

The proof uses simplicial (rather than singular) homology, with a triangulation of $M$ that is adapted to $\Gamma$. While the simplicial approach is convenient for studying a fixed cut, \Cref{def:hom} requires us to compare arbitrary cuts, a flexibility afforded by singular homology.

\begin{proof}
Let $P = \{\Omega_i\}$ denote the partition whose subdomains are the connected components of $M^o\backslash\Gamma$. Triangulating each $\overline{\Omega}_i$ and then subdividing as needed (see \Cref{fig:tri}), we obtain a triangulation of $M$ with the following properties:

\begin{itemize}
	\item The interior of each 2-simplex is contained in some $\Omega_i$.
	\item There is a subset of 1-simplices whose union is $\Gamma$.
	\item Each 1-simplex satisfies exactly one of the following:
	\begin{enumerate}
		\item[(a)] it is a face of two distinct 2-simplices; or
		\item[(b)] it is contained in $\p M$.
	\end{enumerate}
\end{itemize}

Writing the sets of 1- and 2-simplices as $K^{(1)}$ and $K^{(2)}$, respectively, let 
\begin{equation}
\label{gammasum}
	\gamma = \sum_{\tau \in K^{(1)}_\Gamma} \tau, \qquad K^{(1)}_\Gamma = \{\tau \in K^{(1)} : \tau \subset \Gamma\}.
\end{equation}
The isomorphism $H^\Delta_1(M,\p M; \bbZ_2) \to H_1(M,\p M; \bbZ_2)$ between simplicial and singular homology identifies $\gamma$ with a singular 1-chain representing $\Gamma$. Therefore, it suffices to show that 
$P$ is bipartite if and only if $[\gamma] = 0 \in H^\Delta_1(M,\p M; \bbZ_2)$. To prove this, we show that the following are equivalent:
\begin{enumerate}
\item $P$ is bipartite.
\item There exists $K^{(2)}_\Gamma \subset K^{(2)}$ such that
\begin{equation}
\label{AGamma}
	\tau \subset \Gamma \quad \Longleftrightarrow \quad \text{$\tau$ is a face of some $2$-simplices $\sigma \in K^{(2)}_\Gamma$ and $\tilde\sigma \notin K^{(2)}_\Gamma$.}
\end{equation}
\item $[\gamma] = 0 \in H^\Delta_1(M,\p M; \bbZ_2)$.
\end{enumerate}

Note that the condition in (2) means that each 1-simplex in the boundary set $\Gamma$ is the face of two different 2-simplices, exactly one of which is in $K^{(2)}_\Gamma$.

We first prove (1) implies (2). If $P$ is bipartite, then each $\Omega_i$, and hence the interior of each 2-simplex, is colored red or blue. Let $K^{(2)}_\Gamma = \{\sigma \in K^{(2)} : \mbox{int } \sigma \text{ is blue} \}$. If $\tau \subset \Gamma$, then it is not contained in $\p M$, and hence is the face of two distinct 2-simplices, $\sigma$ and $\tilde\sigma$, whose interiors are contained in subdomains $\Omega_i$ and $\Omega_j$. These subdomains are neighbors, and hence have different colors, so exactly one of $\sigma,\tilde\sigma$ is contained in $K^{(2)}_\Gamma$. This proves ($\Rightarrow$) in \eqref{AGamma}, and ($\Leftarrow$) is immediate.

Next, we show that (2) implies (1). If $K^{(2)}_\Gamma \subset K^{(2)}$ satisfies \eqref{AGamma}, we will define a coloring of $P$ by declaring that $\Omega_i$ is ``blue" if it contains the interior of some $\sigma \in K^{(2)}_\Gamma$. To see that this is well defined, we must show that the set $K^{(2)}_{\Omega_i} = \{\sigma \in K^{(2)} : \mbox{int } \sigma \subset \Omega_i\}$ is contained in either $K^{(2)}_\Gamma$ or its complement. If $\sigma, \tilde\sigma \in K^{(2)}_{\Omega_i}$ have a common face that is not contained in $\Gamma$, it follows from \eqref{AGamma} that $\sigma$ and $ \tilde\sigma$ are both in $K^{(2)}_\Gamma$ or both in its complement. Since $\overline{\Omega}_i$ is connected, we conclude that $K^{(2)}_\Gamma$ either contains $K^{(2)}_{\Omega_i}$ or is disjoint from it. 

We now declare $\Omega_i$ to be ``blue" if $K^{(2)}_{\Omega_i} \subset K^{(2)}_\Gamma$ and ``red" otherwise. Suppose $\Omega_i$ and $\Omega_j$ are neighbors, so there exist 2-simplices $\sigma$ and $\tilde\sigma$ with interiors in $\Omega_i$ and $\Omega_j$, respectively, having a common face contained in the boundary set $\Gamma$. It follows from \eqref{AGamma} that one of $\sigma, \tilde\sigma$ is contained in $K^{(2)}_\Gamma$ and the other is not. This means $\Omega_i$ and $\Omega_j$ have different colors, so we conclude that $P$ is bipartite, thus completing the proof that (1) is equivalent to (2).

Next, we show that (2) implies (3). If $K^{(2)}_\Gamma \subset K^{(2)}$ satisfying \eqref{AGamma} exists, we define
\begin{equation}
	\bar \sigma = \sum_{\sigma \in K^{(2)}_\Gamma} \sigma.
\end{equation}
From \eqref{AGamma} we know that every 1-simplex that is not in $\Gamma$ is either in $\p M$ or is the face of an even number of 2-simplices in $K^{(2)}_\Gamma$. It follows that $\p \sigma - \gamma$ is a sum of 1-simplices contained in $\p M$, therefore $\gamma$ is null homologous relative to $\p M$. 

Finally, we show that (3) implies (2). If $[\gamma] = 0 \in H^\Delta_1(M,\p M; \bbZ_2)$, then there is a 2-chain $\bar \sigma$ for which $\p\bar\sigma  - \gamma$ is supported in $\p M$. 
Defining $K^{(2)}_\Gamma$ to be the set of 2-simplices appearing in $\bar\sigma$, we conclude that every 1-simplex in $\gamma$ must be a face of exactly one 2-simplex in $K^{(2)}_\Gamma$, so \eqref{AGamma} holds.
\end{proof}

\subsection{Homology and odd points}
\label{sec:odd}

Given a piecewise $C^1$ cut $\Gamma$, we let $\Gamma_{\rm odd}$ denote the set of points in the interior of $M$ where an odd number of curves terminate. The following observation relates this to \Cref{def:hom}.

\begin{prop}
\label{prop:odd}
If $\Gamma_1$ is homologous to $\Gamma_2$, then $\Gamma_{\rm 1,odd} = \Gamma_{\rm 2,odd}$.
\end{prop}

That is, a necessary condition for $\Gamma_1$ and $\Gamma_2$ to be homologous is that they have the same odd points. This is not, however, a sufficient condition, as shown in \Cref{fig:annulus}. The proposition is an immediate consequence of the following lemma, which also makes it clear why $\Gamma_{\rm 1,odd} = \Gamma_{\rm 2,odd}$ is not sufficient.

\begin{figure}
\begin{center}
\begin{tikzpicture}[ scale=0.7]
	\draw[thick,dashed] (2,0) arc[radius=2, start angle=0, end angle=360];
	\draw[thick,dashed] (1,0) arc[radius=1, start angle=0, end angle=360];
	\draw[ultra thick, blue] (0,1) -- (0,2);
	\draw[ultra thick, blue] (0,-1) -- (0,-2);
	\draw[ultra thick, blue] (1,0) -- (2,0);
	\draw[ultra thick, blue] (-1,0) -- (-2,0);
	\node[blue] at (0,-2.7) {$\Gamma_1$};
\end{tikzpicture}
\hspace{2cm}
\begin{tikzpicture}[ scale=0.7]
	\draw[thick,dashed] (2,0) arc[radius=2, start angle=0, end angle=360];
	\draw[thick,dashed] (1,0) arc[radius=1, start angle=0, end angle=360];
	\draw[ultra thick, red] (0,1) -- (0,2);
	\node[red] at (0,-2.7) {$\Gamma_2$};
\end{tikzpicture}
\hspace{2cm}
\begin{tikzpicture}[ scale=0.7]
	\draw[thick,dashed] (2,0) arc[radius=2, start angle=0, end angle=360];
	\draw[thick,dashed] (1,0) arc[radius=1, start angle=0, end angle=360];
	\draw[ultra thick, magenta] (0,-1) -- (0,-2);
	\draw[ultra thick, magenta] (1,0) -- (2,0);
	\draw[ultra thick, magenta] (-1,0) -- (-2,0);
	\node[magenta] at (0,-2.7) {$\overline{\Gamma_1 \triangle \Gamma_2}$};
\end{tikzpicture}
\end{center}
\caption{Two cuts of the annulus that have $\Gamma_{\rm 1,odd} = \Gamma_{\rm 2,odd} = \varnothing$ but are not homologous.}
\label{fig:annulus}
\end{figure}

\begin{lemma}
\label{lem:cycle}
Let $\gamma_1$ and $\gamma_2$ represent $\Gamma$ and $\Gamma_2$. The difference $\gamma_1 - \gamma_2$ is a relative 1-cycle if and only if $\Gamma_{\rm 1,odd} = \Gamma_{\rm 2,odd}$.
\end{lemma}

In other words, $\Gamma_{\rm 1,odd} = \Gamma_{\rm 2,odd}$ ensures the homology class $[\gamma_1 - \gamma_2] \in H_1(M, \p M; \bbZ_2)$ is defined, but does not guarantee that it is zero except in the case that $H_1(M, \p M; \bbZ_2)$ is trivial. This explains the presence of the annulus in \Cref{fig:annulus}\,---\,no such example is possible on a simply connected domain.

The proof of \Cref{lem:cycle} is a direct consequence of the definition of $\p_1$ and the fact that the endpoints of the curves in $\gamma_1$ and $\gamma_2$ will cancel in pairs, regardless of orientation, since we are using $\bbZ_2$ coefficients.


\section{Homology and unitary equivalence}
\label{sec:unitary}

In this section we prove \Cref{prop:DG}. The case  when $\Gamma_1$ and $\Gamma_2$ intersect transversely uses the equivalent characterization of homology in \Cref{thm:homdef} and is presented in \Cref{sec:Tr}.  The general case is presented in \Cref{sec:Gr} and uses a transversality argument. We prove \Cref{cor:strict} in \Cref{Eq}.

\subsection{The transverse case}\label{sec:Tr}

The following characterization is useful for comparing partitions.

\begin{lemma}
\label{lem:decomp}
If $\Gamma_1$ and $\Gamma_2$ are piecewise $C^1$ cuts that intersect transversely, then 
\begin{align}
\begin{split}\label{H1description}
	W^{1,2}_0(M; \Gamma_1) = \Big\{ u \in W^{1,2}\big(M^o \backslash(\Gamma_1 \cup \Gamma_2)\big) : 
	\text{$u$ vanishes on $\p M$, is continuous} & \\
		\text{across $\Gamma_2$, and is anti-continuous across $\Gamma_1$} &\Big\}.
\end{split}
\end{align}
\end{lemma}

\begin{proof}
Let $u \in W^{1,2}\big(M^o \backslash(\Gamma_1 \cup \Gamma_2)\big)$. Comparing \eqref{H1def} and \eqref{H1description}, it suffices to prove that $u$ is contained in $W^{1,2}(M^o\backslash\Gamma_1)$ if and only if it is continuous across $\Gamma_2$. This follows immediately from the definition of the weak derivative and integration by parts.
\end{proof}

Using this, we prove a special case of \Cref{prop:DG}.

\begin{lemma}
\label{lem:DG}
If $\Gamma_1$ and $\Gamma_2$ are homologous and intersect transversely, then $-\Delta^{\!\Gamma_1}$ and $-\Delta^{\!\Gamma_2}$ are unitarily equivalent. That is, there is a unitary operator $\Phi$ on $L^2(M)$ that maps $W^{1,2}_0(M;\Gamma_1)$ onto $W^{1,2}_0(M; \Gamma_2)$ and satisfies
\[
	b_{_{\Gamma_2}}(\Phi u, \Phi v) = b_{_{\Gamma_1}}(u,v)
\]
for all $u,v \in W^{1,2}_0(M; \Gamma_1)$.

\end{lemma}

\begin{proof}
For convenience we set $\Gamma = \overline{\Gamma_1 \triangle\Gamma_2} = \Gamma_1 \cup \Gamma_2$. Since $\Gamma_1$ and $\Gamma_2$ are homologous, $\Gamma$ is the boundary set of a bipartite partition, by \Cref{thm:homdef}. Denoting this partition by $P = \{\Omega_i\}$, we let $U_+$ and $U_-$ denote the union of the blue and red subdomains, respectively, and then define $\Phi \colon L^2(M) \to L^2(M)$ by
\begin{equation}
\label{def:Phi}
	\Phi u = \begin{cases} u & \text{ on $U_+$}, \\
	-u & \text{ on $U_-$}.
	\end{cases}
\end{equation}
This is an isometric isomorphism of $L^2(M)$ by design.

Suppose $u \in W^{1,2}_0(M; \Gamma_1)$. By \Cref{lem:decomp} we have $u \in W^{1,2}(M^o \backslash(\Gamma_1 \cup \Gamma_2))$, $u$ vanishes on $\p M$, is continuous across $\Gamma_2$, and is anti-continuous across $\Gamma_1$. Since every connected component of $M^o \backslash(\Gamma_1 \cup \Gamma_2)$ is contained in either $U_+$ or $U_-$, we see that $\Phi u$ belongs to $W^{1,2}(M^o \backslash(\Gamma_1 \cup \Gamma_2))$ and still vanishes on $\p M$. Therefore, the desired conclusion will follow by applying \Cref{lem:decomp} with the roles of $\Gamma_1$ and $\Gamma_2$ reversed, if we show that $\Phi u$ is continuous across $\Gamma_1$ and anti-continuous across $\Gamma_2$.

Now consider a smooth segment $\Gamma_* \subset \Gamma$. Since $\Gamma_*$ is contained in $\pO_i \cap \pO_j$ for some $i$ and $j$, $\Omega_i$ and $\Omega_j$ are neighbors. Using that  $P$ is bipartite, $\Omega_i$ and $\Omega_j$  must have different colors, so one is contained in $U_+$ and the other is in $U_-$. It follows that $\Phi$ changes the sign of $u$ on exactly one side of $\Gamma_*$. Recalling that $\Gamma = \overline{\Gamma_1 \Delta \Gamma_2}$, there are two cases to consider:
\begin{enumerate}
	\item $\Gamma_* \subset \Gamma_1$:  \ $u$ is anti-continuous across $\Gamma_*$, therefore $\Phi u$ is continuous across $\Gamma_*$.
	\item $\Gamma_* \subset \Gamma_2$: \ $u$ is continuous across $\Gamma_*$, therefore $\Phi u$ is anti-continuous across $\Gamma_*$.
\end{enumerate}

This implies $\Phi u \in W^{1,2}_0(M; \Gamma_2)$, and it follows immediately that $b_{_{\Gamma_2}}(\Phi u, \Phi v) = b_{_{\Gamma_1}}(u,v)$ for all $u,v \in W^{1,2}_0(M; \Gamma_1)$. The same argument shows that $\Phi^{-1}$ maps $W^{1,2}_0(M; \Gamma_2)$ into $W^{1,2}_0(M; \Gamma_1)$ and completes the proof.
\end{proof}

\subsection{The general case}\label{sec:Gr}

The following result allows us to reduce the proof of \Cref{prop:DG} to the special case already considered in \Cref{lem:DG}.

\begin{prop}
\label{prop:deform}
Given piecewise $C^1$ cuts $\Gamma_1$ and $\Gamma_2$, there is a piecewise $C^1$ cut $\tilde\Gamma$ that is homologous to $\Gamma_1$ and intersects $\Gamma_1$ and $\Gamma_2$ transversely.
\end{prop}

\begin{proof}
We construct $\tilde\Gamma$ by deforming each segment of $\Gamma_1$ via a fixed endpoint homotopy. This guarantees $\tilde\Gamma$ is homologous to $\Gamma_1$. Moreover, we shall see that $\tilde\Gamma$ can be assumed to lie in an arbitrarily small tubular neighborhood of $\Gamma_1$.

Consider a regular $C^1$ curve $\gamma_1 \colon [0,1] \to M$ that is part of $\Gamma_1$. Without loss of generality (subdividing if necessary), we can assume $\gamma_1(0) \neq \gamma_1(1)$. Choosing a normal direction along $\gamma_1$ and letting $t$ denote the distance in this direction, we obtain coordinates $(s,t)$ in a tubular neighborhood of $\gamma_1$, with respect to which $\gamma_1$ is given by $\{ (s,0) : s \in [0,1]\}$. Equivalently, it is the graph of the function $t(s) = 0$.

For sufficiently small $\epsilon>0$, the deformed graph
\begin{equation}
	t_\epsilon(s) = \begin{cases} \epsilon s, & 0 \leq s < \frac12, \\
	\epsilon(1 - s), & \frac12 \leq s \leq 1, \end{cases}
\end{equation}
will be transverse to $\Gamma_1$, $\Gamma_2$ and $\p M$ at both endpoints, and hence for $s$ in some neighborhood $[0,\delta) \cup (1-\delta,1]$. On the other hand, the parametric transversality theorem (see, for instance, \cite[Lem.~II.4.6]{GG73}) guarantees that the ``shifted" curve $t(s) = t_*$ intersects $\Gamma_1$, $\Gamma_2$ and $\p M$ transversely for almost all $t_*$. Therefore, there exists $\delta_* \in (0, \delta)$ such that the graph of
\begin{equation}
\label{perturb}
	\tilde t(s) = \begin{cases} \epsilon s, & 0 \leq s < \delta_*, \\
	\epsilon \delta_*, & \delta_* \leq s \leq 1-\delta_*, \\
	\epsilon(1-s), & 1-\delta_* < s \leq 1,
	\end{cases}
\end{equation}
intersects $\Gamma_1$, $\Gamma_2$ and $\p M$ transversely. It follows immediately from \eqref{perturb} that this curve is homotopic (with fixed endpoints) to $\gamma_1$, and hence is homologous to $\gamma_1$.

To complete the proof we repeat the argument inductively, homotoping each segment of $\Gamma_1$ so that it is transverse to $\Gamma_1$, $\Gamma_2$, $\p M$ and the previously modified segments of $\Gamma_1$.
\end{proof}

\begin{figure}
\begin{tikzpicture}[ scale=1]
	\draw[->] (-0.3,0) -- (3.5,0);
	\draw[->] (0,-0.5) -- (0,1.5);
	\draw[thick] (1,-0.1) -- (1,0.1);
	\draw[thick] (2,-0.1) -- (2,0.1);
	\draw[very thick] (0,0) -- (3,0);
	\draw[very thick, blue] (0,0) -- (1.5,1) -- (3,0);
	\draw[very thick, red] (0,0.5) -- (3,0.5);
	\node[blue] at (2.25,1) {$t_\epsilon$};
	\node[red] at (3.4,0.5) {$t_*$};
	\node at (1,-0.5) {$\scriptstyle{\delta}$};
	\node at (2,-0.5) {$\scriptstyle{1 - \delta}$};
\end{tikzpicture}
\hspace{2cm}
\begin{tikzpicture}[ scale=1]
	\draw[->] (-0.3,0) -- (3.5,0);
	\draw[->] (0,-0.5) -- (0,1.5);
	\draw[thick] (0.75,-0.1) -- (0.75,0.1);
	\draw[thick] (2.25,-0.1) -- (2.25,0.1);
	\draw[very thick] (0,0) -- (3,0);
	\draw[very thick, magenta] (0,0) -- (0.75,0.5) -- (2.25,0.5) -- (3,0);
	\node[magenta] at (1.5,0.9) {$\tilde t$};
	\node at (0.75,-0.5) {$\scriptstyle{\delta_*}$};
	\node at (2.25,-0.5) {$\scriptstyle{1 - \delta_*}$};
\end{tikzpicture}
\caption{The graph of $t_\epsilon$ intersects $\Gamma_1$, $\Gamma_2$ and $\p M$ transversely for $s \in [0,\delta) \cup (1-\delta,1]$, whereas the constant $t_*$ does so for all $s \in [0,1]$. Gluing these together at some $\delta_* \in (0,\delta)$, we obtain a curve that has the same endpoints as $\gamma_1$ and intersects $\Gamma_1$, $\Gamma_2$ and $\p M$ transversely for all $s \in [0,1]$.}
\label{fig:trans}
\end{figure}
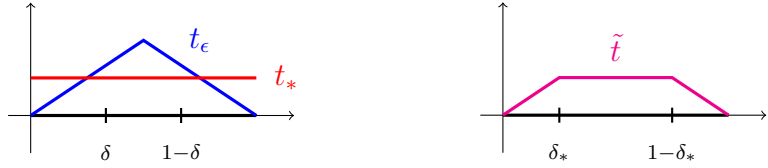

We are finally ready to prove our main proposition.

\begin{proof}[Proof of \Cref{prop:DG}]
Suppose $\Gamma_1$ and $\Gamma_2$ are homologous. Using \Cref{prop:deform} we can find a piecewise $C^1$ cut $\tilde\Gamma$ that is homologous to $\Gamma_1$ (and hence to $\Gamma_2$) and intersects both $\Gamma_1$ and $\Gamma_2$ transversely. Now \Cref{lem:DG} implies that $-\Delta^{\!\Gamma_1}$ and $-\Delta^{\!\Gamma_2}$ are both unitarily equivalent to $-\Delta^{\!\tilde\Gamma}$, so the result follows because unitary equivalence is transitive.
\end{proof}

To recap, we have now established \Cref{prop:DG} and thus \Cref{thm:main} which was derived from \Cref{prop:DG} in \Cref{sec:proof}.

\subsection{The case of equality: proof of \Cref{cor:strict}}
\label{Eq}

\Cref{cor:strict} is an immediate consequence of the following.

\begin{lemma}
\label{lem:equality}
Let $\Gamma$ be a piecewise $C^1$ cut. If $\tilde P$ is the nodal partition of an eigenfunction of $-\DG$ and $\p\tilde P$ is a piecewise $C^1$ cut, then $\p \tilde P$ is homologous to $\Gamma$.
\end{lemma}

In the proof it will be important to keep in mind that an eigenfunction of $-\DG$ does not necessarily change sign across its zero set, but can change sign without being zero, on account of the discontinuity across $\Gamma$. That is, the set where the eigenfunction changes sign need not coincide with its zero set.

\begin{proof}
Using \Cref{prop:deform}, we can obtain a cut $\hat\Gamma$ that is homologous to $\Gamma$ and intersects $\p \tilde P$ transversely. The operators $-\DG$ and $-\Delta^{\!\hat\Gamma}$ are unitarily equivalent by \Cref{prop:DG}. Moreover, the unitary map $\Phi$, defined in \eqref{def:Phi}, preserves zero sets, therefore $\p \tilde P$ is the nodal partition of an eigenfunction of $-\Delta^{\!\hat\Gamma}$.

We thus assume for the rest of the proof that $\Gamma$ intersects $\p \tilde P$ transversely. By \Cref{thm:homdef}, it suffices to show that $\Gamma \cup \p \tilde P$ is the boundary set of a bipartite partition. Let $\psi$ denote the eigenfunction of $-\DG$ whose nodal partition is $\tilde P$. It is only possible for $\psi$ to change sign where it equals zero or is discontinuous, i.e., on the set $\Gamma \cup \p\tilde P$. Since $\Gamma$ and $\p \tilde P$ intersect transversely, there are two possibilities, shown in \Cref{fig:sign}:
\begin{enumerate}
	\item $\p \tilde P \backslash \Gamma$: \ $\psi$ is zero and its normal derivative is continuous, so $\psi$ changes sign;
	\item $\Gamma \backslash \p \tilde P$: \  $\psi$ is nonzero and anti-continuous, so it changes sign;
\end{enumerate}
It follows that $\Gamma \cup \p \tilde P = \overline{\Gamma \triangle \p\tilde P}$ is precisely the set where $\psi$ changes sign, therefore it is the boundary set of a bipartite partition.
\end{proof}

\begin{rem}
\label{rem:case3}
In the proof of \Cref{lem:equality} we only needed to consider the case that $\Gamma$ and $\p\tilde P$ intersect transversely. Nonetheless, it is instructive to consider what happens when $\Gamma \cap \p\tilde P$ contains a curve segment. The eigenfunction $\psi$ will vanish on this segment (because it is contained in the nodal set $\p\tilde P$) and the normal derivative of $\psi$ will be anti-continuous (because it is contained in $\Gamma$), thus $\psi$ will not change sign across this segment. See \Cref{fig:sign}.
\end{rem}

\begin{figure}
\begin{tikzpicture}[ scale=0.8]
	\draw[->] (-2,0) -- (2,0);
	\draw[thick] (0,-0.13) -- (0,0.13);
	\draw[very thick] (-1.5,-1) -- (1.5,1);
	\node at (0,-1.8) {$\p \tilde P \backslash \Gamma$};
\end{tikzpicture}
\hspace{2cm}
\begin{tikzpicture}[ scale=0.8]
	\draw[->] (-2,0) -- (2,0);
	\draw[thick] (0,-0.13) -- (0,0.13);
	\draw[very thick] (-1.5,-0.5) -- (0,-1);
	\draw[very thick] (0,1) -- (1.5,1.5);
	\node at (0,-1.8) {$\Gamma \backslash \p \tilde P$};
\end{tikzpicture}
\hspace{2cm}
\begin{tikzpicture}[ scale=0.8]
	\draw[->] (-2,0) -- (2,0);
	\draw[thick] (0,-0.13) -- (0,0.13);
	\draw[very thick] (-1.5,1) -- (0,0);
	\draw[very thick] (0,0) -- (1.5,1);
	\node at (0,-1.8) {$\Gamma \cap \p \tilde P$};
\end{tikzpicture}
\caption{The graph of $\psi$ along a curve transverse to $\Gamma \cup \p\tilde P$, showing that it only changes sign on the symmetric difference $\overline{\Gamma \triangle \p\tilde P}$, as in \Cref{lem:equality} and \Cref{rem:case3}. }
\label{fig:sign}
\end{figure}
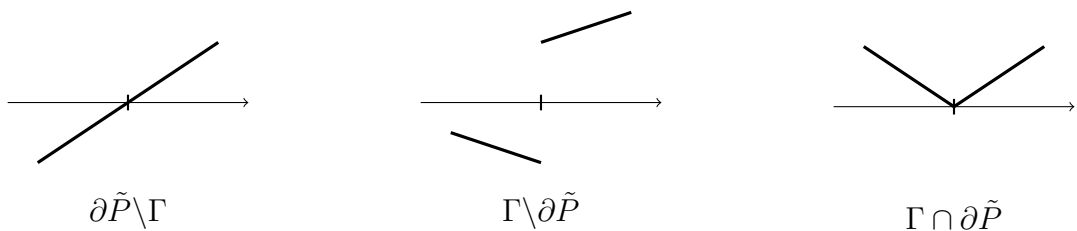


\section{Defining partition eigenvalues for continuous cuts}
\label{sec:cts}

We now relax the regularity hypotheses in \Cref{thm:main}, replacing $\Gamma$ by an arbitrary 1-chain and resulting in the more general \Cref{thm:cts}. To remove the assumption that $\Gamma$ is piecewise $C^1$, we use the fact that the eigenvalues $\{\lambda_n(\Gamma)\}$ of $-\DG$ and the nodal domains of the corresponding eigenfunctions only depend on the homology class of $\Gamma$. This means we can define ``eigenvalues" and ``nodal partitions"  for any singular 1-chain by deforming it to a piecewise $C^1$ cut and then considering the eigenvalues and nodal domains of the corresponding cut Laplacian. Allowing $\p\tilde P$ to be less regular, on the other hand, is more delicate, and requires an approximation argument which is given in \Cref{lem:approx}.

Recalling the notion of a piecewise $C^1$ cut from \Cref{def:regSigma} and a singular 1-chain from \Cref{sec:homdef}, we have the following.

\begin{lemma}
\label{lem:chaincut}
Every singular 1-chain $\gamma \in C_1(M; \bbZ_2)$ is homologous to 
a finite sum of regular $C^1$ curves $\{\tilde\gamma_a\}$ that intersect one another (and $\p M$) transversely, and only do so at their endpoints. Each of these curves can be taken to lie in an arbitrarily small neighborhood of $\gamma$.
\end{lemma}

By a slight abuse of terminology we will say that $\gamma$ is homologous to the piecewise $C^1$ cut generated by the $\tilde\gamma_a$.

\begin{proof}
Write $\gamma = \sum_a \gamma_a$, where each $\gamma_a$ is a continuous curve in $M$. It is standard that each $\gamma_a$ is homotopic (with fixed endpoints) to a smooth curve $\hat\gamma_a$, and this approximation can be done in an arbitrarily small neighborhood of $\gamma_a$. Since there are only finitely many $\hat\gamma_a$, we can deform them one at a time, using \Cref{prop:deform}, so that they intersect one another and $\p M$ transversely. Subdividing if necessary, we can assume that these intersections only occur at the endpoints.
\end{proof}

We can now give our main definition.

\begin{define}
Fix $\gamma \in C_1(M; \bbZ_2)$ and let $\Gamma$ be a piecewise $C^1$ cut homologous to $\gamma$.
\begin{enumerate}
  \item The \emph{eigenvalues of $\gamma$} are the numbers $\lambda_n(\gamma) = \lambda_n(\Gamma)$, $n \geq 1$.
	\item  A partition of $M$ is \emph{$\gamma$-nodal} if it is the nodal partition of an eigenfunction of $-\Delta^{\!\Gamma}$.
    \item A $\gamma$-nodal $k$-partition, corresponding to an eigenvalue $\lambda_*$,  is \emph{Courant sharp} if $k = \min\{n : \lambda_n(\gamma) = \lambda_*\}$.
\end{enumerate}
\end{define}

\Cref{lem:chaincut} guarantees that such a $\Gamma$ exists, and \Cref{prop:DG} ensures that the corresponding eigenvalues and nodal partitions do not depend on the choice of $\Gamma$. It follows that homologous 1-chains $\gamma_1, \gamma_2 \in C_1(M; \bbZ_2)$ have the same eigenvalues and nodal partitions.

Analogous to \eqref{Pkdef}, we define
\begin{equation}
\label{Pkdef2}
	\cP^0_k(\gamma) = \big\{ \tilde P \in \PPk : \p \tilde P \text{ contains the image of a 1-chain that is homologous to } \gamma \big\}.
\end{equation}
Since the curves in a 1-chain are just assumed to be continuous, the boundary set of a partition in $\cP^0_k(\gamma)$ can be quite irregular. Therefore, even if a 1-chain $\gamma$ represents a piecewise $C^1$ cut $\Gamma$, the set $\cP^0_k(\gamma)$ is larger than the set $\cP_k(\Gamma)$ defined in \eqref{Pkdef}.

This means the following result is stronger\footnote{Except when $\Gamma$ is null homologous, in which case $\cP_k(\Gamma) = \cP^0_k(\gamma) = \cP_k$ is the set of all $k$-partitions and the two theorems are equivalent.} than \Cref{thm:main}, no matter how regular $\gamma$ is.

\begin{theorem}
\label{thm:cts}
If $\gamma \in C_1(M; \bbZ_2)$ and $k\in \mathbb N$, then
\begin{equation}
	\lambda_k(\gamma) \leq \inf \big\{\Lambda(\tilde P):\; \tilde P \in \cP^0_k(\gamma) \big\},
\end{equation}
with $\lambda_k(\gamma) = \Lambda(\tilde P)$ if and only if $\tilde P$ is a $\gamma$-nodal partition with eigenvalue $\lambda_k(\gamma)$.
In particular, if $P \in \PPk$ is a Courant sharp $\gamma$-nodal partition, then $\Lambda(P) \leq \Lambda(\tilde P)$ 
for all $\tilde P \in \cP^0_k(\gamma)$. 
\end{theorem}

The proof is similar to that of \Cref{thm:main}, with one crucial difference. The boundary set $\p\tilde P$ contains the image of a 1-chain $\tilde\gamma$ that is homologous to $\gamma$, but since $\tilde\gamma$ is not assumed to represent a piecewise $C^1$ cut, we cannot directly use \Cref{prop:DG} to transform a $\tilde P$-ground state into a test function for the cut Laplacian, as was done above. We instead use the following.

\begin{lemma}
\label{lem:approx}
Let $\gamma \in C_1(M; \bbZ_2)$ and suppose $\tilde P = \{\tilde\Omega_i\} \in \cP^0_k(\gamma)$. For any piecewise $C^1$ cut $\Gamma$ that is homologous to $\gamma$, 
there is a unitary operator $\Phi$ on $L^2(M)$ that maps
$\bigoplus_{i=1}^k W^{1,2}_0(\tilde\Omega_i)$ into $W^{1,2}_0(M;\Gamma)$
and satisfies
\begin{equation}
\label{hat}
	b_{_{\Gamma}}(\Phi u, \Phi u) = \sum_{i=1}^k \int_{\tilde\Omega_i} |\nabla u_i |^2\,dV
\end{equation}
for all $u = (u_1, \ldots, u_k) \in \bigoplus_{i=1}^k W^{1,2}_0(\tilde\Omega_i)$.
\end{lemma}

In the special case that $\p\tilde P$ contains a piecewise $C^1$ cut, say $\tilde\Gamma$, that is homologous to $\Gamma$, then $\oplus_i W^{1,2}_0(\tilde\Omega_i)$ can be identified with a subspace of $W^{1,2}_0(M;\tilde\Gamma)$ and \Cref{lem:approx} follows directly from \Cref{prop:DG}.

\begin{proof}
We first define $\Phi$ on $C^\infty_0(\tilde\Omega_i)$ for a fixed $i$. Let $\{K_n\}$ be a collection of compact sets in $\tilde\Omega_i$ with $K_1 \subset K_2 \subset \cdots$ and $\cup_n K_n = \tilde \Omega_i$, and let $\tilde\Gamma$ denote the subset of $\p\tilde P$ that is homologous to $\Gamma$. For each $n$ we can use \Cref{lem:chaincut} to find a piecewise $C^1$ cut, $\tilde\Gamma_n$, that is homologous to $\tilde\Gamma$ (and hence to $\Gamma$) and is disjoint from $K_n$.

By \Cref{prop:DG} there is a unitary operator $\Phi_n$ on $L^2(M)$ that maps $W^{1,2}_0(M;\tilde\Gamma_n)$ onto $W^{1,2}_0(M;\Gamma)$ and satisfies
\begin{equation}
\label{bGequality}
	b_{_{\Gamma}}(\Phi_n u, \Phi_n v) = b_{_{\tilde\Gamma_n}}(u,v)
\end{equation}
for all $u,v \in W^{1,2}_0(M; \tilde\Gamma_n)$. Changing the signs of some of the $\Phi_n$ if necessary, we can assume that if $\supp u \subset K_m$ for some $m$, then $\Phi_n u = \Phi_m u$ for all $n \geq m$. Thus for any $u_i \in C^\infty_0(\tilde\Omega_i)$ we define
\begin{equation}
\label{iPhi}
    \Phi^i u_i = \lim_{n \to \infty} \Phi_n u_i.
\end{equation}
The limit exists in $W^{1,2}_0(M;\Gamma)$ because it is eventually constant.

Since each $\Phi_n$ is a unitary map on $L^2(M)$, we get $\|\Phi^i u_i \|_{L^2(M)} = \|u_i \|_{L^2(M)}$. Moreover, since $\supp u_i$ will be contained in a single connected component of $M^o\backslash \tilde\Gamma_n$ for large enough $n$, we have from \eqref{bGequality} that
\begin{equation}
\label{Phi0}
	b_{_{\Gamma}}(\Phi^i u_i, \Phi^i u_i) 
    = \lim_{n\to\infty} b_{_{\Gamma}}(\Phi_n u_i, \Phi_n u_i) 
    = \lim_{n\to\infty} b_{_{\tilde\Gamma_n}}(u_i, u_i)
 = \int_{\tilde\Omega_i} |\nabla u_i |^2\,dV.
\end{equation}
By continuity we can therefore extend $\Phi^i$ to a bounded linear map $W^{1,2}_0(\tilde\Omega_i) \to W^{1,2}_0(M;\Gamma)$ that satisfies \eqref{Phi0}. Repeating for each $i$ and letting $\Phi = \oplus_i \Phi_i$ completes the proof.
\end{proof}

With this lemma at our disposal, the rest of the proof of \Cref{thm:cts} is identical to that of \Cref{thm:main}.


\section{Applications}
\label{sec:app}

We now give some applications of \Cref{thm:main} (more precisely, its generalization \Cref{thm:cts}) and \Cref{cor:strict}, one of which was already stated in the introduction. A common theme here is that the topological hypotheses of the theorem are reduced to the computations of relative homology groups.  Note that by Poincar\'e--Lefschetz duality we have $H_1(M,\p M; \bbZ_2) \cong H^1(M; \bbZ_2)$; see, for instance, \cite[Thm.~3.34]{Hat}.

\subsection{The disk}
We consider radial partitions of the disk $M=\{(x,y) : x^2+y^2 \leq 1\}$, as in \Cref{fig:diskrad}. The radial $k$-partition is only Courant sharp for the corresponding partition Laplacian when $1 \leq k \leq 5$. In particular, the radial partitions for $k=1,2,4$ are bipartite and hence minimal, but the minimal 3- and 5-partitions of the disk are not bipartite \cite[Prop.~9.2]{HHOT}. The radial configurations shown in \Cref{fig:diskrad} are each minimal within a certain topological class of partitions, by \Cref{thm:cts}, so they are plausible candidates for the minima.

We now describe the classes in which they are minimal. 
The result for $k=3$ appeared as \Cref{cor:disk} in the introduction; for convenience we restate it here.

\begin{figure}
\begin{tikzpicture}[ scale=0.7]
	\draw[thick,dashed] (2,0) arc[radius=2, start angle=0, end angle=360];
	\draw[ultra thick, blue] (0,0) -- (0,2);
	\draw[ultra thick, blue] (0,0) -- ({2*cos(210)},{2*sin(210)});
	\draw[ultra thick, blue] (0,0) -- ({2*cos(330)},{2*sin(330)});
\end{tikzpicture}
\hspace{2cm}
\begin{tikzpicture}[ scale=0.7]
	\draw[thick,dashed] (2,0) arc[radius=2, start angle=0, end angle=360];
	\draw[ultra thick, blue] (0,0) -- (0,2);
	\draw[ultra thick, blue] (0,0) -- ({2*cos(162)},{2*sin(162)});
	\draw[ultra thick, blue] (0,0) -- ({2*cos(234)},{2*sin(234)});
	\draw[ultra thick, blue] (0,0) -- ({2*cos(306)},{2*sin(306)});
	\draw[ultra thick, blue] (0,0) -- ({2*cos(18)},{2*sin(18)});
\end{tikzpicture}
\caption{Radial 3- and 5-partitions of the disk.}
\label{fig:diskrad}
\end{figure}
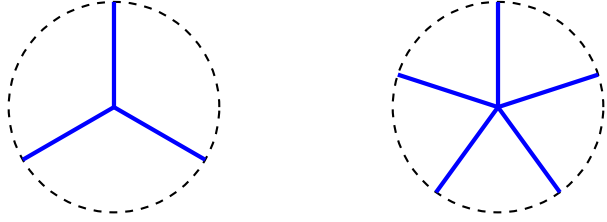

\begin{theorem}
\label{cor:disk35}
The radial $k$-partition, $k\in\{3,5\}$, is minimal among all $k$-partitions whose boundary set contains a curve from the origin to the boundary of the disk. Moreover, it has strictly lower energy than any $k$-partition whose boundary set contains such a curve but is not homologous to it.
\end{theorem}

\begin{proof}
Let $P$ denote the radial $k$-partition in question. This is a nodal partition for $-\DG$, where $\Gamma = \{0\} \times [0,1]$. Suppose $\tilde P$ is a $k$-partition whose nodal set contains a curve from the origin to some point on the boundary. Denoting this curve by $\tilde\Gamma$, we see that the concatenation of $\Gamma$ with $\tilde\Gamma$ is a curve with both endpoints on the boundary. This means $\Gamma - \tilde\Gamma$ is a relative 1-cycle, and hence is null-homologous because $H_1(D, \p D; \bbZ_2)$ is trivial,
so the result follows from \Cref{thm:cts}.
\end{proof}

\subsection{The sphere}

Now consider the 3-partition of $S^2$ given by the meridians $0^\circ$, $120^\circ$ and $-120^\circ$. In \cite{helffer2010spectral} this is called the \textbf{Y}-partition, and it was shown to be the unique (up to rotations) minimal 3-partition of the sphere. To the best of our knowledge, this is the only rigorously known example of a non-bipartite minimal partition. There are two main steps in the proof:
\begin{enumerate}
	\item[(i)] If the boundary set of a minimal 3-partition of $S^2$ contains two antipodal points, then it is (up to rotation) the \textbf{Y}-partition.
	\item[(ii)] The boundary set of any minimal 3-partition of $S^2$ contains two antipodal points.
\end{enumerate}

On the other hand, \Cref{thm:cts} easily implies the following, since $H_1(S^2; \bbZ_2)$ is trivial.

\begin{theorem}
\label{cor:sphere}
The \textbf{Y}-partition is minimal among all 3-partitions whose boundary set contains a curve joining two antipodal points.  Moreover, it has strictly lower energy than any 3-partition whose boundary set contains such a curve but is not homologous to it.
\end{theorem}

In light of the results in \cite{helffer2010spectral}, \Cref{cor:sphere} is not new. However, it provides a different route to the proof of minimality of the \textbf{Y}-partition, by showing that \emph{any} 3-partition whose boundary set contains a curve joining two antipodal points\footnote{If the boundary set of a non-bipartite 3-partition of $S^2$ contains two antipodal points, then it necessarily contains a curve joining them, by \cite[Prop~4.3]{helffer2010spectral}.} has energy at least as big as the \textbf{Y}-partition. That is, \Cref{cor:sphere}  is neither a stronger nor weaker statement than (i), but yields the same conclusion when combined with (ii).

\subsection{The torus}
\label{sec:torus}

Following \cite{HHOtori}, we let $\mathsf{T}(1,b)$ denote the 2-torus obtained by identifying opposite sides of $[0,1] \times [0,b]$, and let $P_k(b)$ denote the $k$-partition of $\mathsf{T}(1,b)$ generated by $k$ uniformly spaced vertical loops, $\{i/k\} \times S^1$.

It is known that $P_k(b)$ is minimal only when $b$ is sufficiently small, so it is natural to ask where the transition occurs, i.e., what is the value of
\begin{equation}
	b_k = \sup \big\{b : P_k(b) \text{ is a minimal partition of $\mathsf{T}(1,b)$} \big\}.
\end{equation}
 In \cite{HHOtori} it was shown that $b_k = 2/k$ if $k$ is even and 
\begin{equation}
\label{bkbound}
	\frac1k \leq b_k \leq \frac{2}{\sqrt{k^2-1}}
\end{equation}
if $k$ is odd. In \cite{BNLtori} the lower bound was improved to $b_k \geq b_k^{\textsf S}$ for some number\footnote{The number $b_k^{\textsf S}$ is defined by an eigenvalue optimization problem on the strip $\bbR \times (0,b)$; see \cite[Eq.~(1.4)]{BNLtori}.}  satisfying $1/k < b_k^{\textsf S} < 1/\sqrt{k^2-1}$, and it was conjectured that $b_k$ is given by the upper bound in \eqref{bkbound}. 

Defining a ``vertical loop" to be a closed curve that winds about $\mathsf{T}(1,b)$ once in the $y$ direction but zero times in the $x$ direction, we get the following from \Cref{thm:cts}.
 
\begin{theorem}
\label{thm:torus}
If $k \geq 3$ is odd and $0 < b \leq 2/\sqrt{k^2-1}$, then $P_k(b)$ is minimal among all $k$-partitions whose boundary set contains a vertical loop.  Moreover, it has strictly lower energy than any $k$-partition whose boundary set contains such a loop but is not homologous to it.
\end{theorem}

\begin{proof}
It is easily verified that $P_k(b)$ is the nodal partition of an eigenfunction of $-\DG$, where $\Gamma = \{0\} \times S^1$, and it is Courant sharp if and only if $b \leq 2/\sqrt{k^2-1}$. To complete the proof we observe that a vertical loop is precisely a curve that is homologous to $\Gamma$.
\end{proof}

Since $P_k(b)$ is a nodal partition, we can conclude from \cite[Thm.~6]{BCCM2} that it locally minimizes the energy $\Lambda$ for $b < 2/\sqrt{k^2-1}$ but not for $b > 2/\sqrt{k^2-1}$. This gives an alternate proof of the upper bound on $b_k$ in \eqref{bkbound}, and in fact gives the stronger result that $P_k(b)$ is not even a \emph{local} minimum for $b$ larger than this value.

A result equivalent to \Cref{thm:torus} was shown in \cite[Prop.~2.7]{BNLtori} by a double covering argument. The equivalence follows from the observation that a $k$-partition of $\mathsf{T}(1,b)$ lifts to a $2k$-partition of the double cover $\mathsf{T}(2,b)$ if and only if its boundary set contains a vertical loop.

\subsection{The cylinder}
The thin cylinder with Neumann boundary conditions was studied in \cite{HHOannulus}, but no results are known for the corresponding Dirichlet problem. We consider $\mathsf{C}(1,b) = S^1 \times (0,b)$, where $S^1$ is the circle with length 1, so the Laplacian eigenfunctions are
\begin{equation}
\label{cylef}
	\cos(2m \pi x) \sin\left(\frac{n\pi y}{b}\right), \ m \geq 0, \ n \geq 1, \qquad 
	\sin(2m \pi x) \sin\left(\frac{n\pi y}{b}\right), \ m, n \geq 1, \qquad 
\end{equation}
with eigenvalues $\lambda_{m,n} = (2m \pi)^2 + (n\pi/b)^2$.

We let $P_k(b)$ denote the $k$-partition of $\mathsf{C}(1,b)$ generated by $k$ uniformly spaced vertical lines, $\{ i/k\} \times [0,b]$. When $k$ is even this is a nodal partition of a Laplacian eigenfunction, namely $\sin(k \pi x) \sin(\pi y/b)$, and it is easily shown that it is Courant sharp, and hence minimal, if and only if $b \leq \sqrt 3/k$.

The case $k$ odd is more interesting.

\begin{theorem}
If $k \geq 3$ is odd and $0 < b \leq \sqrt{3/(k^2-1)}$, then $P_k(b)$ is minimal among all $k$-partitions whose nodal set contains a curve between the two boundary circles.  Moreover, it has strictly lower energy than any $k$-partition whose boundary set contains such a curve but is not homologous to it.
\end{theorem}

On the other hand, it follows from \cite[Thm.~6]{BCCM2} that for $b > \sqrt{3/(k^2-1)}$ the partition $P_k(b)$ is not even a local minimum.

\begin{proof}
Consider $-\DG$ with $\Gamma = \{0\} \times S^1$. The eigenfunctions are given by \eqref{cylef} but with $m = \frac12, \frac32, \frac52, \ldots$ and $n \in \bbN$, therefore $P_k(b)$ is the nodal partition of the eigenfunction $\sin(k \pi x) \sin(\pi y/b)$, which has $m = \frac{k}{2}$ and $n=1$. To see when this is Courant sharp, we note that $\lambda_{k/2,1}$ is the $k$th eigenvalue of $-\DG$ if and only if $\lambda_{k/2,1} \leq \lambda_{1/2,2}$, which is equivalent to $k^2 + b^{-2} \leq 1 + 4 b^{-2}$ and hence to $b \leq  \sqrt{3/(k^2-1)}$. 
Finally, we note that $H_1(M,\p M; \bbZ_2) = \bbZ_2$ is generated by a curve between the inner and outer boundaries, therefore any such curve must be homologous to $\Gamma$.
\end{proof}

\begin{rem}
In \cite{HHOannulus} it is shown that the nodal set of a minimal 3-partition of a sufficiently thin cylinder, with Neumann conditions on the external boundary, must contain a curve between the two boundary circles. It is essential to the proof that the groundstates of the partition satisfy mixed Dirichlet/Neumann boundary conditions (on the internal/external boundary, respectively), therefore the argument does not work for the Dirichlet problem.
\end{rem}

\subsection*{Acknowledgments}
The authors thank Tom Baird, Bernard Helffer and Peter Kuchment for helpful comments and discussions.  G.B. acknowledges the support of NSF Grant DMS-2247473.  Y.C. was supported by NSF CAREER Grant DMS-2045494. G.C. acknowledges the support of NSERC grant RGPIN-2017-04259.  J.L.M. acknowledges support from the NSF through  NSF Applied Math Grant DMS-2307384 and NSF FRG grant DMS-2152289.  The authors are grateful to the AIM SQuaRE program for hosting them and supporting the initiation of this project.

\bibliographystyle{siam}
\bibliography{nodal}

\end{document}